\documentclass[11pt]{article} 

\usepackage[utf8]{inputenc}

\usepackage{graphicx}
\usepackage[ruled,vlined]{algorithm2e}
\usepackage{algorithmic}
\usepackage{epsfig}
\usepackage{epstopdf}
\usepackage{amssymb,amsmath}
\usepackage{amsmath}
\usepackage{amsthm}
\usepackage{amsfonts}
\usepackage{psfrag}
\usepackage{rotating}
\usepackage{latexsym}
\usepackage{dsfont}
\usepackage{stmaryrd}
\usepackage{amsthm}
\usepackage{amssymb}
\usepackage{amsmath}
\usepackage{mathtools}
\usepackage{mathrsfs}   
\usepackage{siunitx}
\usepackage{pgfpages}
\usepackage[Euler]{upgreek} 
\usepackage{isomath}
\usepackage{physics}
\usepackage{pifont} 
\usepackage{bbm}
\usepackage{empheq}
\usepackage[thicklines]{cancel}
\usepackage{subcaption}
\usepackage[authoryear,semicolon]{natbib}

\usepackage[force,almostfull]{textcomp}
\usepackage{fancyvrb}
\usepackage{textcomp}
\usepackage{url}
\usepackage{ifdraft}
\usepackage{url}
\usepackage{multirow}
\usepackage{rotating}
\usepackage{xspace}
\usepackage{array}
\usepackage{multido}
\usepackage{enumerate}
\usepackage[utf8]{inputenc}
\newlength\figureheight 
\newlength\figurewidth 
\usepackage{graphics}
\usepackage{pgfplots}
\pgfplotsset{compat=newest}
\pgfplotsset{plot coordinates/math parser=false}
\allowdisplaybreaks

\usepackage{scalefnt}

\newtheoremstyle{specialcasestyle}{1mm}{1mm}{\upshape}{}{\bfseries\upshape}{.}{0mm}{}
\theoremstyle{specialcasestyle}

\newtheorem{assump}{Assumption}

\newtheorem{prop}{Proposition}
\newtheorem{rem}{Remark}
\newtheorem{thm}{Theorem}

\theoremstyle{remark}          

\newcommand*{\Var}[1]{\ensuremath{\mathrm{Var}\left[#1\right]}}
\newcommand*{\E}[1]{\ensuremath{\mathbb{E}\left[#1\right]}}
\newcommand*{\prob}[1]{\ensuremath{\mathbb{P}\left[#1\right]}}
\newcommand*{\tol}{\ensuremath{\mathrm{TOL}}}

\newcommand\eqdef{\stackrel{\mathclap{\tiny\mbox{def}}}{=}}

\begin{document}

\title{{Multilevel Importance Sampling for Rare Events Associated With the McKean--Vlasov Equation}}
\author{Nadhir Ben Rached  \thanks{Department of Statistics, School of Mathematics, University of Leeds, UK ({\tt N.BenRached@leeds.ac.uk}).}, Abdul-Lateef Haji-Ali \thanks{Department of Actuarial Mathematics and Statistics, School of Mathematical and Computer Sciences, Heriot-Watt University, Edinburgh, UK ({\tt A.HajiAli@hw.ac.uk}).}, Shyam Mohan Subbiah Pillai \thanks{Corresponding author; Chair of Mathematics for Uncertainty Quantification, Department of Mathematics, RWTH Aachen University, Aachen, Germany({\tt subbiah@uq.rwth-aachen.de}).} \\
and Ra\'ul Tempone  \thanks{Computer, Electrical and Mathematical Sciences \& Engineering Division (CEMSE), King Abdullah University of Science and Technology (KAUST), Thuwal, Saudi Arabia ({\tt raul.tempone@kaust.edu.sa}). Alexander von Humboldt Professor in Mathematics for Uncertainty Quantification, RWTH Aachen University, Aachen, Germany ({\tt tempone@uq.rwth-aachen.de}).}.
        \thanks{This work was supported by the KAUST Office of Sponsored Research (OSR) under Award No. URF/1/2584-01-01 and the Alexander von Humboldt Foundation. This work was also partially performed as part of the Helmholtz School for Data Science in Life, Earth and Energy (HDS-LEE) and received funding from the Helmholtz Association of German Research Centres. For the purpose of open access, the author has applied a Creative Commons Attribution (CC BY) license to any author-accepted manuscript version arising from this submission.}
}
\date{}
\maketitle
\thispagestyle{empty}
\begin{abstract}
	This work combines multilevel Monte Carlo (MLMC) with importance sampling to estimate rare-event quantities that can be expressed as the expectation of a Lipschitz observable of the solution to {a broad class of McKean--Vlasov stochastic differential equations}. We extend the double loop Monte Carlo (DLMC) estimator introduced in this context in \citep{my_paper} to the multilevel setting. We formulate a novel {multilevel DLMC} estimator and perform a comprehensive {cost-error} analysis yielding new and improved complexity results. Crucially, we devise an antithetic sampler to estimate level differences guaranteeing reduced {computational complexity} for the {multilevel DLMC} estimator compared with the single-level DLMC estimator. To address rare events, we apply the importance sampling scheme, obtained via stochastic optimal control in~\citep{my_paper}, over all levels of the {multilevel DLMC} estimator. Combining importance sampling and {multilevel DLMC} reduces {computational complexity} by one order and drastically reduces the associated constant compared to the single-level DLMC estimator without importance sampling. We illustrate the effectiveness of the proposed {multilevel DLMC} estimator on the Kuramoto model from statistical physics with Lipschitz observables, confirming the reduced complexity from $\order{\tol_{\mathrm{r}}^{-4}}$ for the single-level DLMC estimator to $\order{\tol_{\mathrm{r}}^{-3}}$ while providing a feasible estimate of rare-event quantities up to prescribed relative error tolerance $\tol_{\mathrm{r}}$. 
\end{abstract}

\section{Introduction}
\label{sec:intro}

We consider the estimation of rare-event quantities expressed as an expectation of some observable of the solution to {a broad class of McKean--Vlasov stochastic differential equations (MV-SDEs)}. In particular, we develop a computationally efficient multilevel Monte Carlo (MLMC) estimator for $\E{G(X(T))}$, where $G:\mathbb{R}^d \rightarrow \mathbb{R}$ is a Lipschitz function and $X:[0, T]\cross\Omega \rightarrow \mathbb{R}^d$ is the MV-SDE process up to a finite terminal time $T$. {We consider $G$ to be a rare-event observable, meaning that the support of $G$ lies deep in the tails of the distribution of $X(T)$.} MV-SDEs are a special class of SDEs whose drift and diffusion coefficients are a function of the law of the solution~\citep{mckean_vlasov}. Such equations arise from the mean-field behavior of stochastic interacting particle systems used in diverse applications~\citep{pedestrian_app,animal_app,chemical_app,combustion_app,biology_app,finance_app}. Significant recent literature has addressed the analysis~\citep{mvsde_existence,mvsde_pde,mvsde_smoothpde} and numerical treatment~\citep{mlmc_mvsde,mvsde_timescheme,mvsde_iterative,mvsde_cubature} of MV-SDEs. MV-SDEs are often approximated using stochastic $P$-particle systems, sets of $P$ coupled $d$-dimensional SDEs that approach a mean-field limit as the number of particles tends to infinity~\citep{mean_field_limit}. The associated Kolmogorov backward equation is $P d$-dimensional; hence, Monte Carlo (MC)  methods are used to approximate expectations associated with particle systems. MC methods using Euler--Maruyama time-discretized particle systems for bounded, Lipschitz drift/diffusion coefficients have been proposed for smooth, nonrare observables with a complexity of $\order{\tol^{-4}}$ for a prescribed error tolerance $\tol$~\citep{ogawa1992monte,mlmc_mvsde,euler_mvsde}. 

{MLMC} was introduced as an improvement to MC for SDEs in ~\citep{mlmc_sde}. {MLMC} is based on generalized control variates and improves efficiency when an approximation of the solution is computed based on a discretization parameter~\citep{giles_mlmc}. Most {MLMC} simulations are performed at cheaper, coarse levels, with relatively few simulations applied at the costlier, fine levels. {MLMC} for particle systems has been widely investigated~\citep{mlmc_mvsde,mlmc_mvsde_2,mlmc_mvsde_3,mlmc_mvsde_4}. In particular, an {MLMC} estimator with a partitioning sampler achieved a {computational complexity} of $\order{\tol^{-3}}$ for smooth observables and bounded, Lipschitz drift/diffusion coefficients~\citep{mlmc_mvsde}. However, MC and {MLMC} methods become extremely expensive in the context of rare events due to the ‘blowing up’ of the constant associated with the estimator complexity as the event becomes rarer~\citep{is_general_ref}. This motivates using importance sampling as a variance reduction technique to overcome the failure of standard MC and {MLMC} in the rare-event regime~\citep{is_general_ref}. 

Importance sampling for MV-SDEs has been studied in~\citep{is_mvsde,my_paper}. The decoupling approach developed by~\citep{is_mvsde} defines a modified, decoupled MV-SDE with coefficients computed using a realization of the MV-SDE law estimated beforehand using a stochastic particle system. A change of measure is applied to the decoupled MV-SDE, decoupling importance sampling from the law estimation. The theory of large deviations and Pontryagin principle were employed in~\citep{is_mvsde} to obtain a deterministic, time-dependent control minimizing a proxy for the estimator variance. \citep{my_paper} employed the same decoupling approach to define a double loop MC (DLMC) estimator and employ stochastic optimal control to derive a time- and state-dependent control minimizing the variance of the importance sampling estimator. In~\citep{my_paper}, an adaptive DLMC algorithm with a complexity of $\order{\tol^{-4}}$ was developed, the same as that for the MC estimator for nonrare observables in~\citep{mlmc_mvsde}, while enabling feasible estimates for rare-event probabilities. The development of such an importance sampling scheme using stochastic optimal control theory has been proposed before in other contexts, including standard SDEs~\citep{is_entropy,is_diffusions,is_cross_entropy}, stochastic reaction networks~\citep{soc_srn}, and discrete-time continuous-space Markov chains~\citep{soc_sumrvs,is_markov_chains}.

We combine importance sampling with {MLMC} to reduce the relative estimator variance in the rare-event regime by extending the results in~\citep{my_paper} to the multilevel setting. Combining importance sampling with {MLMC} has been previously explored in other contexts, including standard SDEs~\citep{mlmcis_alaya,mlmcis_kebaier,mlmcis_giles} and stochastic reaction networks~\citep{mlmcis_srn}. We extend the DLMC estimator from~\citep{my_paper} to the multilevel setting by introducing a {multilevel DLMC} estimator along with a detailed error and complexity analysis. We boost the efficiency of the {multilevel DLMC} estimator by developing a highly correlated, antithetic sampler for level differences~\citep{mlmc_mvsde,mlmc_antithetic}. We show reduced {computational complexity} using the {multilevel DLMC} estimator compared with the DLMC estimator for MV-SDEs. Then, we propose an importance sampling scheme for this estimator based on variance minimization using stochastic optimal control theory~\citep{my_paper} to address rare events and apply the obtained control on all levels. Contributions of this paper are summarized as follows.

\begin{itemize}
	\item We extend the DLMC estimator introduced by~\citep{my_paper} to the multilevel setting and propose a {multilevel DLMC} estimator for the decoupling approach~\citep{is_mvsde} for MV-SDEs. We include a detailed discussion on the bias and variance of the proposed estimator and devise a complexity theorem, displaying improved complexity compared with the naïve DLMC estimator. We also formulate a robust {multilevel DLMC} algorithm to determine optimal parameters adaptively.
	\item We develop naïve and antithetic correlated samplers for level differences in the {multilevel DLMC} estimator. Numerical simulations confirm increased variance convergence rates for level-difference estimators using the antithetic sampler compared with the naïve one, leading to improved {multilevel DLMC} estimator complexity.
	\item We propose combining importance sampling with the {multilevel DLMC} estimator to address rare events. We employ the time- and state-dependent control developed by~\citep{my_paper} for all levels in the {multilevel DLMC} estimator. Numerical simulations confirm a significant variance reduction in the {multilevel DLMC} estimator due to this importance sampling scheme, improving estimator complexity from $\order{\tol_\mathrm{r}^{-4}}$ in~\citep{my_paper} for the Kuramoto model with Lipschitz observables to $\order{\tol_\mathrm{r}^{-3}}$ in the multilevel setting while allowing feasible rare-event quantity estimation up to the prescribed relative error tolerance $\tol_\mathrm{r}$.
\end{itemize}

The remainder of this paper is structured as follows. In Section~\ref{sec:mvsde}, we introduce the MV-SDE and associated notation, motivate MC methods to estimate expectations associated with its solution and set forth the problem to be solved. In  {Section~\ref{sec:dlmc}, we introduce the decoupling approach for MV-SDEs~\citep{is_mvsde} and formulate a DLMC estimator.} Next, we state the optimal importance sampling control for the decoupled MV-SDE derived using stochastic optimal control and introduce the DLMC estimator with importance sampling from~\citep{my_paper} in Section~\ref{sec:is_dlmc}. Then, we  introduce the novel {multilevel DLMC} estimator in Section~\ref{sec:mldlmc}, develop an antithetic sampler for it, and derive new complexity results for the estimator. We combine the {multilevel DLMC} estimator with the proposed importance sampling scheme and develop an adaptive {multilevel DLMC} algorithm that feasibly estimates rare-event quantities associated with MV-SDEs. Finally, we apply the proposed methods to the Kuramoto model from statistical physics in Section~\ref{sec:results} and numerically verify all assumptions in this work and the derived complexity rates for the {multilevel DLMC} estimator for two observables. 


\section{McKean--Vlasov Stochastic Differential Equation}
\label{sec:mvsde}

{In this work, we consider a broad class of McKean--Vlasov equations that arise from the mean-field limit of stochastic interacting particle systems with pairwise interaction kernels~\citep{mean_field_limit}.} We consider the probability space $\{\Omega,\mathcal{F},\{\mathcal{F}_t\}_{t \geq 0},P\}$, where $\mathcal{F}_t$ is the filtration of a standard Wiener process. For functions $b:\mathbb{R}^d \cross \mathbb{R} \longrightarrow \mathbb{R}^d$, $\sigma:\mathbb{R}^d \cross \mathbb{R} \longrightarrow \mathbb{R}^{d \cross d}$, $\kappa_1: \mathbb{R}^d \cross \mathbb{R}^d \longrightarrow \mathbb{R}$, and $\kappa_2: \mathbb{R}^d \cross \mathbb{R}^d \longrightarrow \mathbb{R}$, we consider the following It\^o SDE for the McKean--Vlasov stochastic process $X: [0,T] \cross \Omega \rightarrow \mathbb{R}^d$:
\begin{empheq}[left=\empheqlbrace, right = ,]{equation} 
    \label{eqn:mvsde}
    \begin{alignedat}{2}
        \dd X(t) &= b\left(X(t),\int_{\mathbb{R}^d} \kappa_1 (X(t),x) \mu_t(\dd x)\right) \dd t  \\
        &\qquad + \sigma \left(X(t),\int_{\mathbb{R}^d} \kappa_2 (X(t),x) \mu_t(\dd x)\right) \dd W(t), \quad t>0 \\
        X(0) &= x_0 \sim \mu_0 \in \mathcal{P}(\mathbb{R}^d) , 
    \end{alignedat}
\end{empheq}
where $W:[0,T] \cross \Omega \rightarrow \mathbb{R}^d$ is a standard $d$-dimensional Wiener process with mutually independent components. $\mu_t \in \mathcal{P}(\mathbb{R}^d)$ is the probability distribution of $X(t)$, where $\mathcal{P}(\mathbb{R}^d)$ is the space of probability measures on $\mathbb{R}^d$. $x_0 \in \mathbb{R}^d$ is a random initial state with distribution $\mu_0 \in \mathcal{P}(\mathbb{R}^d)$. The functions $b$ and $\sigma$ are called drift and diffusion functions/coefficients, respectively. {The existence and uniqueness of solutions to \eqref{eqn:mvsde} follows from the results in~\citep{mvsde_existence,mvsde_soln_theory,mean_field_limit,mvsde_weak_soln}, under the assumptions therein. These involve certain differentiability and boundedness conditions on $b$, $\sigma$, $\kappa_1$, and $\kappa_2$.} The time evolution of $\mu_t$ is given by the multidimensional Fokker--Planck partial differential equation (PDE): 
\begin{empheq}[left=\empheqlbrace, right = ,]{alignat=2}
    \label{eqn:fokkerplanck_mvsde}
    &-\frac{\partial \mu(s,x;t,y)}{\partial s} - \sum_{i=1}^d \frac{\partial}{\partial x_i} \left(b_i\left(x,\int_{\mathbb{R}^d} \kappa_1(x,z) \mu(s,z;t,y) \dd z \right) \mu(s,x;t,y)\right) \nonumber \\
    &+ \sum_{i=1}^d \sum_{j=1}^d \frac{1}{2}\frac{\partial^2}{\partial x_i \partial x_j} \Bigg( \Bigg. \sum_{k=1}^d \sigma_{ik} \sigma_{jk} \left(x,\int_{\mathbb{R}^d} \kappa_2(x,z) \mu(s,z;t,y) \dd z \right) \nonumber \\
    &\qquad \mu(s,x;t,y) \Bigg. \Bigg) = 0, \quad (s,x) \in (t,\infty) \cross \mathbb{R}^d  \\
    &\mu(t,x;t,y) = \delta_y(x) ,  \nonumber
\end{empheq}
where $\mu(s,x;t,y)$ denotes the conditional distribution of $X(s)$ given that $X(t) = y$, and $\delta_y(\cdot)$ denotes the Dirac measure at $y$. Equation~\eqref{eqn:fokkerplanck_mvsde} is a nonlinear integral PDE with nonlocal terms. Solving such an equation using classical numerical methods up to relative error tolerances can be computationally prohibitive, particularly in higher dimensions ($d \gg 1$). 

A strong approximation of the solution to the above class of MV-SDEs can be obtained by solving a system of $P$ exchangeable It\^o SDEs, also known as a stochastic interacting particle system, with pairwise interaction kernels comprising $P$ particles~\citep{mean_field_limit}. For $p=1, \ldots, P$, we have the following SDE for the process $X^P_p:[0,T] \cross \Omega \rightarrow \mathbb{R}^d$:
\begin{empheq}[left=\empheqlbrace, right = ,]{alignat=2}
    \label{eqn:strong_approx_mvsde}
    \dd X^P_p(t) &= b\left(X^P_p(t), \frac{1}{P} \sum_{j=1}^P  \kappa_1(X^P_p(t),X^P_j(t)) \right) \dd t \nonumber \\
    &\qquad + \sigma\left(X^P_p(t), \frac{1}{P} \sum_{j=1}^P \kappa_2(X^P_p(t),X^P_j(t)) \right) \dd W_p(t), \quad t>0 \\
    X^P_p(0) &= (x_0)_p \sim \mu_0 \in \mathcal{P}(\mathbb{R}^d) , \nonumber
\end{empheq}
where $\{(x_0)_p\}_{p=1}^P$ are independent and identically distributed (i.i.d.) random variables sampled from the initial distribution $\mu_0$, and $\{W_p\}_{p=1}^P$ are mutually independent $d$-dimensional Wiener processes also independent of $\{(x_0)_p\}_{p=1}^P$. Equation~\eqref{eqn:strong_approx_mvsde} approximates the mean-field distribution $\mu_t$ from \eqref{eqn:mvsde} using an empirical distribution based on particles $\{X^P_p\}_{p=1}^P$:
\begin{equation}
    \label{eqn:emp_dist_law}
    \mu_t(\dd x) \approx \mu_t^P(\dd x) = \frac{1}{P} \sum_{j=1}^P \delta_{X^P_j(t)} (\dd x) , 
\end{equation}
where particles $\{X^P_p\}_{p=1}^P$ are identically distributed but not mutually independent due to pairwise interaction kernels in the drift and diffusion coefficients. {The strong convergence of particle systems follows from the results given in~\citep{mvsde_strong_conv_1,mvsde_strong_conv_2,mvsde_strong_conv_3}, under further assumptions on the drift and diffusion coefficients therein.} The high dimensionality of the Fokker--Planck equation, satisfied by the joint probability density of the particle system, motivates the use of MC methods, which do not suffer from the curse of dimensionality.

 


\subsection{Example: fully connected Kuramoto model for synchronized oscillators}  \label{sec:kuramoto}

We focus on a one-dimensional (1D) example of the MV-SDE in \eqref{eqn:mvsde}, called the Kuramoto model, which describes synchronization in statistical physics to help model the behavior of large sets of coupled oscillators. This model has widespread applications in chemical and biological systems \citep{chemical_app}, neuroscience \citep{neuroscience_app}, and oscillating flame dynamics \citep{combustion_app}. In particular, the Kuramoto model is a system of $P$ fully connected, synchronized oscillators. We consider the following Itô SDE for the process $X_p^P:[0,T] \cross \Omega \rightarrow \mathbb{R}$:
\begin{empheq}[left=\empheqlbrace, right = ,]{alignat=2}
    \label{eqn:kuramoto_model}
    \dd X^P_p(t) &= \left({\xi_p} + \frac{1}{P} \sum_{q=1}^P \sin\left(X^P_p(t) - X^P_q(t)\right)\right) \dd t + \sigma \dd W_p (t) , \quad t>0\\
    X^P_p(0) &= (x_0)_p \sim \mu_0 \in \mathcal{P}(\mathbb{R}) , \nonumber
\end{empheq}
where {$\{\xi_p\}_{p=1}^P$} denotes i.i.d. random variables sampled from a prescribed distribution. The diffusion $\sigma \in \mathbb{R}$ is constant, and $\{(x_0)_p\}_{p=1}^P$ represents i.i.d. random variables sampled from a prescribed distribution $\mu_0$. In addition, $\{W_p\}_{p=1}^P$ represents mutually independent 1D Wiener processes, and {$\{\xi_p\}_{p=1}^P$}, $\{(x_0)_p\}_{p=1}^P$, and $\{W_p\}_{p=1}^P$ are mutually independent. This coupled particle system reaches the mean-field limit as the number of oscillators tends to infinity. In this limit, each particle satisfies the following MV-SDE: 
\begin{empheq}[left=\empheqlbrace, right = ,]{equation}
    \label{eqn:kuramoto_mvsde}
    \begin{alignedat}{2}
        \dd X(t) &= \left({\xi} + \int_{\mathbb{R}} \sin (X(t)-x) \mu_t (\dd x) \right) \dd t + \sigma \dd W(t), \quad t>0 \\
        X(0) &= x_0 \sim \mu_0 \in \mathcal{P}(\mathbb{R}) , 
    \end{alignedat}
\end{empheq}
where $X(t)$ denotes the state of each particle at time $t$, {$\xi$} represents a random variable sampled from some prescribed distribution, and $\mu_t$ is the probability distribution of $X(t)$. This example is used throughout this work as a test case for the proposed MC algorithms.

\subsection{Problem setting}
\label{sec:setting}

We let $T>0$ be some finite terminal time and $X:[0,T] \cross \Omega \rightarrow \mathbb{R}^d$ denote the McKean-Vlasov process \eqref{eqn:mvsde}. We let $G: \mathbb{R}^d \longrightarrow \mathbb{R}$ be a given Lipschitz rare-event observable. The objective is to build a computationally efficient {MLMC} estimator $\mathcal{A}_{\textrm{MLMC}}$ for $\E{G(X(T))}$ with a given relative tolerance $\tol_\mathrm{r} > 0$ that satisfies
{
\begin{equation}
    \label{eqn:mlmc_objective}
    \prob{\frac{\abs{\mathcal{A}_{\textrm{MLMC}}-\E{G(X(T))}}}{\abs{\E{G(X(T))}}} \geq \tol_\mathrm{r}} \leq \nu,
\end{equation}
for a given confidence level determined by $0 < \nu \ll 1$. 
}
The high dimensionality of the Kolmogorov backward equation corresponding to the stochastic particle system \eqref{eqn:strong_approx_mvsde} makes numerical solutions of $\E{G(X(T))}$ up to some $\tol_{\mathrm{r}}$ computationally infeasible. This motivates employing MC methods to overcome the curse of dimensionality. Combining MC with variance reduction techniques, such as importance sampling, is required to produce feasible rare-event estimates. 

In~\citep{my_paper}, we introduced the DLMC estimator, based on a decoupling approach \citep{is_mvsde} to provide a simple importance sampling scheme implementation minimizing the estimator variance. The current paper extends the DLMC estimator to the multilevel setting, achieving better complexity than $\order{\tol_\mathrm{r}^{-4}}$ from the single-level DLMC estimator. {Section~\ref{sec:dlmc} introduces the decoupling approach for MV-SDEs and associated notation before introducing the DLMC estimator.}

{
\section{Double Loop Monte Carlo Estimator Using a Decoupling Approach}
\label{sec:dlmc}

The decoupling approach was developed for importance sampling in MV-SDEs~\citep{is_mvsde,my_paper}, where the idea is to approximate the MV-SDE law empirically as in \eqref{eqn:emp_dist_law}, use the approximation as input to define a decoupled MV-SDE and apply a change of measure to it. We decouple the computation of the MV-SDE law and the change in probability measure required for importance sampling. First, we introduce the general decoupling approach.

\subsection{Decoupling Approach for McKean--Vlasov Stochastic Differential Equation}
\label{sec:decoupling}

The decoupling approach in~\citep{my_paper,is_mvsde} comprises the following steps.

\begin{enumerate}
	\item We approximate the MV-SDE law $\{\mu_t: t \in [0, T]\}$ using the empirical measure $\{\mu_t^P: t \in [0, T]\}$ in \eqref{eqn:emp_dist_law} using particles $\{X^P_p(t): t \in [0, T]\}_{p=1}^P$ satisfying \eqref{eqn:strong_approx_mvsde}.

    \item Given $\{\mu_t^P: t \in [0,T]\}$, we define the decoupled MV-SDE for the process $\Bar{X}^P:[0,T] \times \Omega \rightarrow \mathbb{R}^d$ as
    \begin{empheq}[left=\empheqlbrace, right = ,]{alignat=2}
        \label{eqn:decoupled_mvsde}
        \begin{split}
            &\dd \Bar{X}^P(t) = b\left(\Bar{X}^P(t), \frac{1}{P} \sum_{j=1}^P  \kappa_1(\Bar{X}^P(t),X^P_j(t)) \right) \dd t \\
            &\qquad + \sigma \left(\Bar{X}^P(t), \frac{1}{P} \sum_{j=1}^P  \kappa_2(\Bar{X}^P(t),X^P_j(t)) \right) \dd \Bar{W}(t), \quad t \in [0,T] \\
            &\Bar{X}^P(0) = \Bar{x}_0 \sim \mu_0, \quad \Bar{x}_0 \in \mathbb{R}^d , 
        \end{split}
    \end{empheq}
    where superscript $P$ indicates that the drift and diffusion functions in \eqref{eqn:decoupled_mvsde} are computed using $\{\mu_t^P: t \in [0,T]\}$ derived from the stochastic $P$-particle system. Drift and diffusion coefficients $b$ and $\sigma$ are the same as defined in Section~\ref{sec:mvsde}. In addition, $\Bar{W}:[0,T]\cross \mathbb{R}^d \rightarrow \mathbb{R}^d$ is a standard $d$-dimensional Wiener process independent of the Wiener processes $\{W_p\}_{p=1}^P$ used in \eqref{eqn:strong_approx_mvsde}, and $\Bar{x}_0 \in \mathbb{R}^d$ is a random initial state sampled from $\mu_0$ as defined in \eqref{eqn:mvsde} and is independent from $\{(x_0)_p\}_{p=1}^P$ in \eqref{eqn:strong_approx_mvsde}. Thus, \eqref{eqn:decoupled_mvsde} is a standard SDE with random coefficients. 
    
	\item We introduce a copy space (see~\citep{is_mvsde}) to distinguish the decoupled MV-SDE \eqref{eqn:decoupled_mvsde} from the stochastic $P$-particle system. We suppose \eqref{eqn:strong_approx_mvsde} is defined on the probability space $(\Omega,\mathcal{F},\mathbb{P})$. We define a copy space $(\Bar{\Omega},\Bar{\mathcal{F}},\Bar{\mathbb{P}})$; hence, we define \eqref{eqn:decoupled_mvsde} on the product space $(\Omega,\mathcal{F},\mathbb{P}) \cross (\Bar{\Omega},\Bar{\mathcal{F}},\Bar{\mathbb{P}})$. Thus, $\mathbb{P}$ is a probability measure generated by the randomness of $\{\mu_t^P: t \in [0, T]\}$, and $\Bar{\mathbb{P}}$ denotes the measure generated by the randomness of the Wiener process driving \eqref{eqn:decoupled_mvsde} conditioned on $\{\mu_t^P:t\in [0, T]\}$.
    
    \item We approximate the quantity of interest as 
    \begin{align}
        \label{eqn:total_exp_mvsde}
        \E{G(X(T))} & \approx \mathbb{E}_{\mathbb{P} \otimes \Bar{\mathbb{P}}} \left[G(\Bar{X}^P(T))\right] \nonumber \\
        &= \mathbb{E}_\mathbb{P} \left[\mathbb{E}_{\Bar{\mathbb{P}}} \left[ G(\Bar{X}^P(T)) \mid \{\mu_t^P: t \in [0,T]\} \right]\right] \cdot
    \end{align}
    Henceforth, we omit the probability measure above to simplify the notation such that $\E{G(\Bar{X}^P(T))}$ means the expectation is taken with respect to all randomness in the decoupled MV-SDE \eqref{eqn:decoupled_mvsde}. We estimate the inner expectation $\E{G(\Bar{X}^P(T))\mid \{\mu_t^P: t \in [0, T]\}}$ for a given $\{\mu_t^P: t \in [0, T]\}$ and then estimate the outer expectation using MC sampling over different realizations of $\{\mu_t^P: t \in [0, T]\}$.
    
\end{enumerate}

The inner expectation $\mathbb{E}\left[G(\Bar{X}^P(T))\mid \{\mu_t^P: t \in [0,T]\}\right]$ can be obtained using the Kolmogorov backward equation for the decoupled MV-SDE \eqref{eqn:decoupled_mvsde}. Obtaining an analytical solution to the Kolmogorov backward equation is not always possible, and conventional numerical methods do not handle relative error tolerances, which are relevant in the rare-event regime. This motivates MC methods coupled with importance sampling, even for the 1D case, to estimate the nested expectation~\eqref{eqn:total_exp_mvsde} in the rare-event regime. We approximate the nested expectation~\eqref{eqn:total_exp_mvsde} using a DLMC estimator, as in~\citep{my_paper}. The general outline of the DLMC estimator is given in Algorithm~\ref{alg:dlmc_outline}. In the following, we use the notation $\omega_{p:P}^{(i)} \eqdef \left( \omega_q^{(i)} \right)_{q=p}^P$, where for each $q$, $\omega_q^{(i)}$ denotes the $i^\text{th}$ sample of the set of underlying random variables that are used in calculating the empirical law $\mu_t^{P}$. Hence, $\mu^{P}_t(\omega_{1:P}^{(i)})$ denotes the $i^{\mathrm{th}}$ realization of the empirical law \eqref{eqn:emp_dist_law}. We let $\bar{\omega}^{(i)}$ denote the $i^{\mathrm{th}}$ realization of random variables driving the decoupled MV-SDE dynamics \eqref{eqn:decoupled_mvsde} conditioned on a realization of the empirical law.  

\begin{algorithm}
    \caption{Outline of the double loop Monte Carlo algorithm for decoupled McKean--Vlasov stochastic differential equation}
\label{alg:dlmc_outline}
    \SetAlgoLined
    \textbf{Inputs: } $P,M_{1},M_2$; \\
    \For{$i=1,\ldots,M_1$}{
    Generate $\left\{\mu^{P}_t(\omega_{1:P}^{(i)}):t \in [0,T]\right\}$ using \eqref{eqn:strong_approx_mvsde},\eqref{eqn:emp_dist_law}; \\
    \For{$j=1,\ldots,M_2$}{
    Given $\left\{\mu^{P}_t(\omega_{1:P}^{(i)}):t \in [0,T]\right\}$, generate sample path of \eqref{eqn:decoupled_mvsde} using $\bar{\omega}^{(j)}$;\\
    Compute $G\left(\Bar{X}^{P}(T)\right) \left(\omega_{1:P}^{(i)}, \bar{\omega}^{(j)}\right)$; \\
    }
    }
    Approximate $\E{G\left(X(T)\right)}$ by $\frac{1}{M_1} \sum_{i=1}^{M_1} \frac{1}{M_2} \sum_{j=1}^{M_2} G\left(\Bar{X}^{P}(T)\right) \left(\omega_{1:P}^{(i)}, \bar{\omega}^{(j)}\right)$ ;
\end{algorithm}

The decoupled MV-SDE~\eqref{eqn:decoupled_mvsde} for the given empirical law $\left\{\mu^{P}_t:t \in [0, T]\right\}$ is a standard SDE, making it possible to use stochastic optimal control to derive an optimal change of measure minimizing the variance of the estimator of the inner expectation $\E{G\left(\Bar{X}^{P}(T)\right) \mid \left\{\mu^{P}_t:t \in [0, T]\right\}}$, as formulated in previous studies~\citep{is_entropy,is_diffusions,is_cross_entropy}. Such an importance sampling scheme for the decoupled MV-SDE was derived in~\citep{my_paper} and is summarized in Section~\ref{sec:is_dlmc}.
}

\section{Importance Sampling for the Decoupled McKean--Vlasov Stochastic Differential Equation}
\label{sec:is_dlmc}

{
This section applies stochastic optimal control theory to obtain the optimal change of measure for the decoupled MV-SDE~\eqref{eqn:decoupled_mvsde}. Then, we incorporate the above importance sampling scheme to the DLMC Algorithm~\ref{alg:dlmc_outline}, and formulate the DLMC estimator with importance sampling.
}
\subsection{Optimal Importance Sampling Control for Decoupled McKean--Vlasov Stochastic Differential Equation}
\label{sec:optimal_is}

\citep{my_paper} derived an optimal change of measure for the decoupled MV-SDE, minimizing the MC estimator variance based on stochastic optimal control. {This was based on the well-known Girsanov theorem for change of measure~\citep{sde_oksendal} for standard SDEs.} We recall the main results here. First, we formulate the Hamilton--Jacobi--Bellman control PDE that provides optimal control for the decoupled MV-SDE. We introduce the following notation: $\langle \cdot,\cdot \rangle$ is the Euclidean dot product between two functions in $\mathbb{R}^d$, $\nabla \cdot$ denotes the gradient vector of a scalar function, $\nabla^2 \cdot$ represents the {Hessian matrix} of a scalar function, $\cdot: \cdot$ indicates the Frobenius inner product between two matrix-valued functions, $\norm{\cdot}$ is the Euclidean norm of a function in $\mathbb{R}^d$ and $C^k(A,B)$ denotes the set of real-valued bounded continuous functions from domain $A$ to set $B$ with $k \geq 1$ bounded continuous derivatives on $B$.  

\begin{prop}[Hamilton--Jacobi--Bellman PDE for decoupled MV-SDE~\citep{my_paper}]
    \label{th:hjb_decoupled_mvsde}
    \hspace{1mm} Let the process $\Bar{X}^P$ follow the dynamics \eqref{eqn:decoupled_mvsde}. We consider the following Itô SDE for the controlled process $\Bar{X}^P_\zeta: [0,T] \cross \Omega \rightarrow \mathbb{R}^d$ with control $\zeta: [0,T] \cross \mathbb{R}^d \rightarrow \mathbb{R}^d$:
   	\begin{empheq}[left=\empheqlbrace, right = \cdot]{alignat=2}
        \label{eqn:dmvsde_sde_is}
        \begin{split}
            \dd \Bar{X}^P_\zeta(t) &= \Bigg( \Bigg. b\left(\Bar{X}^P_\zeta(t), \frac{1}{P} \sum_{j=1}^P  \kappa_1(\Bar{X}^P_\zeta(t),X^P_j(t)) \right) \\
            &+ \sigma \left(\Bar{X}^P_\zeta(t), \frac{1}{P} \sum_{j=1}^P  \kappa_2(\Bar{X}^P_\zeta(t),X^P_j(t)) \right) \zeta(t,\Bar{X}^P_\zeta(t)) \Bigg. \Bigg) \dd t \\
            &+ \sigma \left(\Bar{X}^P_\zeta(t), \frac{1}{P} \sum_{j=1}^P  \kappa_2(\Bar{X}^P_\zeta(t),X^P_j(t)) \right) \dd W(t), \quad 0<t<T \\ 
            \Bar{X}^P_\zeta(0) &= \Bar{X}^P(0) = \Bar{x}_0 \sim \mu_0 . 
        \end{split}
    \end{empheq}
    where \eqref{eqn:strong_approx_mvsde} is used to compute $\{\mu^P_t:t \in [0,T]\}$ in \eqref{eqn:decoupled_mvsde} and \eqref{eqn:dmvsde_sde_is}. The value function $u:[0,T] \cross \mathbb{R}^d \rightarrow \mathbb{R}^d$ minimizing the second moment (for the derivation, see~\citep{my_paper}) {of the DLMC estimator with importance sampling} is written as 
    \begin{align}
        \label{eqn:dmvsde_value_fxn}
        u(t,x) &= \min_{\zeta \in \mathcal{Z}}      \mathbb{E}\Bigg[\Bigg.G^2(\Bar{X}_\zeta^P(T)) \exp{-\int_t^T \norm{\zeta(s,\Bar{X}_\zeta^P(s))}^2 - 2 \int_t^T \langle \zeta(s,\Bar{X}_\zeta^P(s)),\dd W(s) \rangle} \nonumber\\
        &\qquad \Big| \quad \Bar{X}_\zeta^P(t) = x, \{\mu_t^P: t \in [0,T]\}\Bigg.\Bigg] \cdot
    \end{align} 
    {Here $\mathcal{Z} = \left\{ f : f \in C^1 \left( [0,T] \cross \mathbb{R}^d , \mathbb{R}^d \right) \right\}$ is the set of admissible deterministic $d$-dimensional Markov controls. Assume $b$ and $\sigma$ are sufficiently regular such that $u$ has bounded and continuous derivatives up to first order in time and second order in space and $u(t,x) \neq 0 \quad \forall (t,x) \in [0,T] \cross \mathbb{R}^d$.} We define $\gamma:[0,T] \cross \mathbb{R}^d \rightarrow \mathbb{R}^d$, such that $u(t,x) = \exp{-2 \gamma(t,x)}$. Then, $\gamma$ satisfies the nonlinear Hamilton--Jacobi--Bellman equation:
    \begin{empheq}[left=\empheqlbrace, right = ,]{alignat=2}
        \label{eqn:dmvsde_hjb_form2}
        \begin{split}
            &\frac{\partial \gamma}{\partial t} + \langle b\left(x, \frac{1}{P} \sum_{j=1}^P  \kappa_1(x,X^P_j(t)) \right), \nabla \gamma \rangle + \frac{1}{2} \nabla^2 \gamma : \left(\sigma \sigma^T\right) \left(x, \frac{1}{P} \sum_{j=1}^P  \kappa_2(x,X^P_j(t)) \right)\\
            & - \frac{1}{4} \norm{\sigma^T \nabla \gamma \left(x, \frac{1}{P} \sum_{j=1}^P  \kappa_2(x,X^P_j(t)) \right)}^2 = 0, \quad (t,x) \in [0,T) \cross \mathbb{R}^d \\
            &\gamma(T,x) = - \log \abs{G(x)}, \quad x \in \mathbb{R}^d , 
        \end{split}
    \end{empheq}
    with optimal control 
    \begin{equation}
        \zeta^*(t,x) = - \sigma^T \left(x, \frac{1}{P} \sum_{j=1}^P  \kappa_2(x,X^P_j(t)) \right) \nabla \gamma \left(t,x \right),
    \end{equation}
    minimizing the second moment.
\end{prop}

\begin{proof}
	See Appendix B in~\citep{my_paper}.
\end{proof}

{Previous studies~\citep{my_paper,zero_variance_is} have demonstrated that \eqref{eqn:dmvsde_hjb_form2} leads to a zero-variance estimator of the inner expectation $\E{G\left(\Bar{X}^{P}(T)\right) \mid \left\{\mu^{P}_t:t \in [0,T]\right\}}$, provided $G(\cdot)$ does not change sign. Using the transformation $u(t,x) = v^2(t,x)$, we can recover the linear Kolmogorov backward equation for the dynamics \eqref{eqn:decoupled_mvsde} for a given $\left\{\mu^{P}_t:t \in [0,T]\right\}$.}

\begin{empheq}[left=\empheqlbrace, right = ,]{alignat=2}
    \label{eqn:dmvsde_hjb_form3}
    \begin{split}
        &\frac{\partial v}{\partial t} + \langle b\left(x, \frac{1}{P} \sum_{j=1}^P  \kappa_1(x,X^P_j(t)) \right), \nabla v \rangle \\
        & \qquad + \frac{1}{2} \nabla^2 v : \left(\sigma \sigma^T\right) \left(x, \frac{1}{P} \sum_{j=1}^P  \kappa_2(x,X^P_j(t)) \right) = 0 ,\quad (t,x) \in [0,T) \cross \mathbb{R}^d \\
        &v(T,x) = \abs{G(x)}, \quad x \in \mathbb{R}^d , 
    \end{split}
\end{empheq}
with optimal control 
\begin{equation}
    \label{eqn:dmvsde_hjb_optimal_control}
    \zeta^*(t,x) = \sigma^T \left(x, \frac{1}{P} \sum_{j=1}^P  \kappa_2(x,X^P_j(t)) \right) \nabla \log v (t,x) \cdot
\end{equation}

{
\begin{rem}
	\hspace{1mm} In Proposition~\ref{th:hjb_decoupled_mvsde}, we control the decoupled McKean--Vlasov process $\bar{X}^P$ (whose dynamics are given in~\eqref{eqn:decoupled_mvsde}) instead of the particles $X^P_p$ from the interacting particle system~\eqref{eqn:strong_approx_mvsde} because the optimal control problem would be $d$-dimensional instead of $P d$-dimensional.
\end{rem}
}
We require a realization of the empirical law from the stochastic $P$-particle system \eqref{eqn:strong_approx_mvsde} to obtain the control using \eqref{eqn:dmvsde_hjb_form3} and \eqref{eqn:dmvsde_hjb_optimal_control}. In practice, we obtain a time-discretized version of the empirical law using the Euler--Maruyama scheme. {To avoid computing the optimal control for each $\{\mu_t^P: t \in [0, T]\}$ realization in the DLMC estimator, we independently obtain a sufficiently accurate empirical law realization off-line using a sufficiently large number of particles and time steps (see~\citep{my_paper}, Algorithm~2). This approach is motivated by the convergence of the empirical law to the MV-SDE law $\{\mu_t: t \in [0, T]\}$ as the number of particles and time steps tend to infinity. This approach has two main advantages: we do not need to solve the original KBE \eqref{eqn:dmvsde_hjb_form3} for the value function $u$ to satisfy relative tolerance $\tol_\mathrm{r}$, and we don't require high accuracy for the solution to $\zeta^*$ for the purpose of importance sampling.} Section~\ref{sec:results} confirms that the deterministic control thus obtained is sufficient to ensure variance reduction in the proposed estimator.
{
\begin{rem}
\label{rem:hjb_dimension}
    \hspace{1mm} As a proof of concept, we numerically solve the 1D ($d=1$) control PDE arising from the Kuramoto model \eqref{eqn:kuramoto_model} using finite differences and extend the solution to the entire domain using linear interpolation. Using such a method to solve \eqref{eqn:dmvsde_hjb_form3} in higher dimensions ($d \gg 1$) is computationally expensive due to the curse of dimensionality. In such cases, model reduction techniques~\citep{is_model_reduction,is_projection} or solving the minimization problem \eqref{eqn:dmvsde_value_fxn} using stochastic gradient methods~\citep{is_entropy} may help. We do not focus on building efficient methods to solve the control PDE in this work.
\end{rem}
}
\subsection{Double Loop Monte Carlo Estimator With Importance Sampling}
\label{sec:dlmc_single}

We briefly outline the single-level DLMC estimator for a given importance sampling control $\zeta$ obtained off-line by solving \eqref{eqn:dmvsde_hjb_optimal_control} (for more details, see \citep{my_paper}).
{
\begin{enumerate}
    \item 
   		We consider the discretization $0=t_0<t_1<t_2<\ldots<t_{N} = T$ of the time domain $[0, T]$ with $N$ equal time steps of the particle system \eqref{eqn:strong_approx_mvsde} (i.e., $t_n = n \Delta t, \quad n=0,1,\ldots,N$ and $\Delta t = T/{N}$). The discretized version of particle $X_p^P$ corresponding to \eqref{eqn:strong_approx_mvsde} with $P$ particles is denoted by $X_p^{P|N}$. Let $\omega_{1:P}$ denote the $P$ underlying sets of random variables used to generate a realization of $X_p^{P|N}$ from \eqref{eqn:strong_approx_mvsde}.
     \item We define the discrete law obtained from the time-discretized particle system by $\mu^{P|N}$ as 
    \begin{equation}
        \label{eqn:dmvsde_discrete_law}
        \mu^{P|N}(t_n) = \frac{1}{P} \sum_{p=1}^P \delta_{X_p^{P|N}(t_n)}, \quad \forall n=0,\ldots,N \cdot
    \end{equation}
    Then, we define a time-continuous extension of the empirical law by extending the time-discrete stochastic particle system to all $t \in [0, T]$ using the continuous-time forward Euler-Maruyama extension.
    \item Given the approximate law $\mu^{P|N}$ from \eqref{eqn:dmvsde_discrete_law} and control $\zeta:[0,T] \times \mathbb{R}^d \rightarrow \mathbb{R}^d$, we generate sample paths $\{\Bar{X}^{P|N}_\zeta(t):t \in [0,T]\}$ of the controlled decoupled MV-SDE \eqref{eqn:dmvsde_sde_is}. Let $\bar{\omega}$ denote the set of random variables to generate one of the above sample paths.
    \item We consider the same discretization of the time domain $[0, T]$ as the particle system~\eqref{eqn:strong_approx_mvsde} for the controlled decoupled MV-SDE~\eqref{eqn:dmvsde_sde_is} with $N$ equal time steps. We define $\{\Bar{X}^{P|N}_\zeta(t_n)\}_{n=1}^{N}$ as the Euler--Maruyama time-discretized version of $\Bar{X}^{P|N}_\zeta$, the decoupled MV-SDE process \eqref{eqn:dmvsde_sde_is} defined using empirical law $\mu^{P|N}$ \eqref{eqn:dmvsde_discrete_law}.
    \item Thus, we can express the quantity of interest with importance sampling as follows:
    \begin{equation}
        \label{eqn:dmvsde_qoi_is}
        \E{G(\Bar{X}^{P|N}(T))} = \E{G(\Bar{X}_\zeta^{P|N}(T))\mathbb{L}^{P|N}}.
    \end{equation}
    This expectation is approximated using the DLMC estimator $\mathcal{A}_{\textrm{MC}}$ from~\citep{my_paper}.
    \begin{equation}
        \label{eqn:dmvsde_dlmc_est}
        \mathcal{A}_{\textrm{MC}} = \frac{1}{M_1} \sum_{i=1}^{M_1} \frac{1}{M_2} \sum_{j=1}^{M_2} G \left(\Bar{X}_\zeta^{P|N}(T) \right) \mathbb{L}^{P|N} \left( \omega_{1:P}^{(i)}, \bar{\omega}^{(j)} \right) \cdot
    \end{equation}
    where the likelihood factor (for detailed derivation, see~\citep{my_paper}) is 
    \begin{align}
        \label{eqn:dlmc_llhood_factor}
        \mathbb{L}^{P|N}\left( \omega_{1:P}^{(i)}, \bar{\omega}^{(j)} \right) &= \prod_{n=0}^{N-1} \exp \Bigg\{ \Bigg. -\frac{1}{2} \Delta t \norm{\zeta(t_n,\Bar{X}_\zeta^{P|N}(t_n))}^2 \\
        &\quad - \langle \Delta W(t_n), \zeta(t_n,\Bar{X}^{P|N}_\zeta(t_n)) \rangle \Bigg. \Bigg\} \left( \omega_{1:P}^{(i)}, \bar{\omega}^{(j)} \right), \nonumber
    \end{align}
    Here, $M_1$ is the number of realizations of $\mu^{P|N}$ in the DLMC estimator, and $\omega_{1:P}^{(i)}$ denotes the $i^{\mathrm{th}}$ realization of $\omega_{1:P}$. For each realization of $\mu^{P|N}$, $M_2$ is the number of sample paths for the decoupled MV-SDE, and $\bar{\omega}^{(j)}$ denotes the $j^{\mathrm{th}}$ realization of $\bar{\omega}$. Further $\Delta W(t_n) \sim \mathcal{N}(0,\sqrt{\Delta t}\mathbb{I}_d)$ are the Wiener increments driving the dynamics of the time-discretized decoupled MV-SDE~\eqref{eqn:dmvsde_sde_is}.    
\end{enumerate}

\begin{rem}
    \hspace{1mm} The MC estimator for smooth, nonrare observables based on the particle system approximation introduced by~\citep{mlmc_mvsde} has a {computational complexity} of $\order{\tol^{-4}}$ for a given absolute error tolerance $\tol$. However, the constant associated with this complexity substantially increases in the rare-event regime~\citep{is_general_ref}. This problem can be overcome using the above importance sampling scheme. \citep{my_paper} demonstrates that the DLMC estimator \eqref{eqn:dmvsde_dlmc_est} with importance sampling has a complexity of $\order{\tol_{\mathrm{r}}^{-4}}$ for estimating rare-event probabilities up to the prescribed relative error tolerance $\tol_{\mathrm{r}}$. Additionally, importance sampling ensures that the constant associated with the complexity of this DLMC estimator \eqref{eqn:dmvsde_dlmc_est} is significantly reduced, enabling a feasible computation of rare-event probabilities.
\end{rem}
}
Section \ref{sec:mldlmc} extends this estimator to the multilevel setting to obtain better complexity.

\section{Multilevel Double Loop Monte Carlo Estimator With Importance Sampling}
\label{sec:mldlmc}

\subsection{Multilevel Double Loop Monte Carlo Estimator for Decoupled McKean--Vlasov Stochastic Differential Equation}
\label{sec:mldlmc_est}

{Two discretization parameters ($P$ and $N$) approximate the solution to the decoupled MV-SDE. For {MLMC} purposes, we introduce the parameter (level) $\ell$ that couples both discretization parameters. The geometric sequence of levels is defined given the parameter $\tau$.} For $\ell=0,1,\ldots,L$, let:

\begin{align}
    P_\ell &= P_0 \tau^\ell, \nonumber\\
    \label{eqn:mldlmc_hierarchy}
    N_\ell &= N_0 \tau^{\ell} \cdot
\end{align}

We set $G = G(X(T))$, and its corresponding discretization at level $\ell$ is {$G_\ell = G(\Bar{X}^{P_\ell|N_\ell}(T))$, where $X:[0,T] \cross \Omega \rightarrow \mathbb{R}^d$ is the MV-SDE process \eqref{eqn:mvsde} and $\Bar{X}^{P_\ell|N_\ell}$ satisfies \eqref{eqn:decoupled_mvsde}.}
The {MLMC} concept~\citep{giles_mlmc} uses a telescoping sum for level $L\in \mathbb{N}$:

\begin{equation}
    \label{eqn:mldlmc_telescop}
    \mathbb{E}[G_L] = \sum_{\ell=0}^L \mathbb{E}[G_\ell - G_{\ell-1}] \quad, G_{-1}=0 \cdot
\end{equation} 

The current work approximates each expectation in \eqref{eqn:mldlmc_telescop} independently using the DLMC method, creating the {multilevel DLMC} estimator:
\begin{equation}
    \label{eqn:mldlmc_estimator}
    \mathcal{A}_{\textrm{MLMC}} = \sum_{\ell=0}^L \frac{1}{M_{1,\ell}} \sum_{i=1}^{M_{1,\ell}} \frac{1}{M_{2,\ell}} \sum_{j=1}^{M_{2,\ell}} (G_\ell - \mathcal{G}_{\ell-1})(\omega_{1:P_\ell}^{(\ell,i)},\bar{\omega}^{(\ell,j)}) \cdot
\end{equation}
{
where $\mathcal{G}_{\ell-1}$ is a random variable correlated to $G_\ell$ such that $\mathcal{G}_{-1} = 0$ and $\E{\mathcal{G}_{\ell-1}} = \E{G_{\ell-1}}$, ensuring $\E{\mathcal{A}_{\textrm{MLMC}}} = \E{G_L}$. In addition, $\omega^{(\ell,i)}_{1:P_\ell}$ refers to the $i^{\mathrm{th}}$ realization of the $P_\ell$ underlying sets of random variables used to estimate $\mu^{P_\ell|N_\ell}$ at level $\ell$, and $\bar{\omega}^{(\ell,j)}$ denotes the $j^{\mathrm{th}}$ realization of the random variables in~\eqref{eqn:decoupled_mvsde} for the given $\mu^{P_\ell|N_\ell}$ realization at level $\ell$.} Samples of $G_\ell(\omega_{1:P_\ell}^{(\ell,i)},\bar{\omega}^{(\ell,j)})$ and $\mathcal{G}_{\ell-1}(\omega_{1:P_\ell}^{(\ell,i)},\bar{\omega}^{(\ell,j)})$ must be sufficiently correlated to ensure that the {multilevel DLMC} estimator has better complexity than the single-level DLMC estimator. We explore two correlated samplers motivated by the {MLMC} estimators in~\citep{mlmc_mvsde}.

\begin{enumerate}
	 \item \emph{Naïve Sampler}. We use the first $P_{\ell-1}$ random variables out of each set of $P_\ell$ random variables to obtain empirical law $\mu_{\ell-1} = \mu^{P_{\ell-1} | N_{\ell-1}}$ to generate a sample of $\mathcal{G}_{\ell-1}$ using \eqref{eqn:strong_approx_mvsde}. Given $\mu_{\ell-1}$, we solve \eqref{eqn:decoupled_mvsde} at level $\ell-1$ using the same $\bar{\omega}$ as for level $\ell$ and compute the quantity of interest: 
    \begin{equation}
        \label{eqn:naive_sampler}
        \mathcal{G}_{\ell-1}(\omega_{1:P_\ell}^{(\ell,i)},\bar{\omega}^{(\ell,j)}) = \bar{\mathcal{G}}_{\ell-1}(\omega_{1:P_\ell}^{(\ell,i)},\bar{\omega}^{(\ell,j)}) \quad {\eqdef} \quad G_{\ell-1}(\omega_{1:P_{\ell-1}}^{(\ell,i)},\bar{\omega}^{(\ell,j)}) \cdot
    \end{equation}
    \item \emph{Antithetic Sampler}. We split the $P_\ell$ sets of random variables into $\tau$ i.i.d. groups of size $P_{\ell-1}$ each. Then, to generate a sample of $\mathcal{G}_{\ell-1}$, we use each group to independently simulate \eqref{eqn:strong_approx_mvsde} and obtain empirical law $\mu^{(a)}_{\ell-1}$ for $a=1,\ldots,\tau$. Given $\mu^{(a)}_{\ell-1}$, we solve \eqref{eqn:decoupled_mvsde} at level $\ell-1$ independently for each $a$ using the same $\bar{\omega}$ as for level $\ell$. The quantity of interest is computed for each group and averaged over the $\tau$ groups: 
    \begin{equation}
        \label{eqn:antithetic_sampler}
        \mathcal{G}_{\ell-1}(\omega_{1:P_\ell}^{(\ell,i)},\bar{\omega}^{(\ell,j)}) = \hat{\mathcal{G}}_{\ell-1}(\omega_{1:P_\ell}^{(\ell,i)},\bar{\omega}^{(\ell,j)}) \quad {\eqdef} \quad \frac{1}{\tau} \sum_{a=1}^{\tau} G_{\ell-1}(\omega_{(a-1)P_{\ell-1}+1:aP_{\ell-1}}^{(\ell,i)},\bar{\omega}^{(\ell,j)}) \cdot
    \end{equation}
\end{enumerate}

Section~\ref{sec:results} numerically investigates the effects of these two correlation schemes on the variance convergence for level differences. 

\subsubsection{Error analysis}

We aim to build an efficient {multilevel DLMC} estimator that satisfies \eqref{eqn:mlmc_objective}. We can bound the global relative error of $\mathcal{A}_{\textrm{MLMC}}$ as follows:
\begin{equation}
    \label{eqn:mldlmc_global_err}
    \frac{\abs{\E{G}-\mathcal{A}_{\textrm{MLMC}}}}{\abs{\E{G}}} \leq \underbrace{\frac{\abs{\E{G} - \E{G_L}}}{\abs{\E{G}}}}_{=\epsilon_b \text{, Relative bias}} + \underbrace{\frac{\abs{\E{G_L} - \mathcal{A}_{\textrm{MLMC}}}}{\abs{\E{G}}}}_{=\epsilon_s} ,
\end{equation}

We impose more restrictive error constraints than \eqref{eqn:mlmc_objective},
\begin{align}
    \label{eqn:mldlmc_bias_constraint}
    &\epsilon_b \leq \theta \tol_\mathrm{r} \abs{\E{G}}, \\
    \label{eqn:mldlmc_stat_constraint}
    &{\prob{\epsilon_s \geq (1-\theta)\tol_\mathrm{r} \abs{\E{G}}} \leq \nu} ,
\end{align}
for a given tolerance splitting parameter $\theta \in (0,1)$. We make the following assumption for the bias.
\begin{assump}[{Multilevel DLMC} Bias]
    \label{ass:mldlmc_bias}
    \hspace{2mm} There exists a constant $\tilde{\alpha}>0$ such that
    \begin{equation*}
        \epsilon_b = \frac{\abs{\E{G} - \E{G_\ell}}}{\abs{\E{G}}} \lesssim  \tau^{-\tilde{\alpha} \ell} ,
    \end{equation*}
\end{assump}
where $\lesssim$ indicates that a constant $C$ exists independent of $\ell$, such that $\epsilon_b \leq C \tau^{-\tilde{\alpha} \ell}$, and constant $\Tilde{\alpha}$ denotes the bias convergence rate with respect to level $\ell$. Section~\ref{sec:results} verifies this assumption and determines $\Tilde{\alpha}$ numerically for the Kuramoto model. {Assumption~\ref{ass:mldlmc_bias} is motivated by the weak convergence with respect to the number of particles~\citep{kolokoltsov2019mean} and with respect to the number of time steps~\citep{sde_numerics}. In order to use these bounds in this work, we assume the drift/diffusion coefficients satisfy further regularity and boundedness conditions entailed in~\citep{kolokoltsov2019mean,sde_numerics}.} We set the level $L$ to satisfy the bias constraint \eqref{eqn:mldlmc_bias_constraint}. {The statistical error constraint~\eqref{eqn:mldlmc_stat_constraint} can be rewritten as follows:

\begin{equation}
	\label{eqn:stat_normality}
	\prob{\epsilon_s \geq (1-\theta)\tol_\mathrm{r} \abs{\E{G}}} = \prob{\frac{\epsilon_s}{\sqrt{\Var{\mathcal{A}_{\textrm{MLMC}}}}} \geq \frac{(1-\theta) \tol_\mathrm{r} \abs{\E{G}}}{\sqrt{\Var{\mathcal{A}_{\textrm{MLMC}}}}}} \leq \nu \cdot
\end{equation}

By assuming the normality (at least asymptotically) of the estimator $\mathcal{A}_{\textrm{MLMC}}$, we obtain the following condition for the estimator variance:

\begin{equation}
    \label{eqn:mldlmc_var_constraint}
    \Var{\mathcal{A}_{\textrm{MLMC}}} \leq \left(\frac{(1-\theta)\tol_\mathrm{r} \abs{\E{G}}}{C_\nu}\right)^2,
\end{equation}
where $C_\nu$ is the $\left(1-\frac{\nu}{2}\right)$-quantile for the standard normal distribution. The asymptotic normality of {MLMC} estimators can be demonstrated using the Lindeberg--Feller central limit theorem~\citep{lf_clt_theorem}.} The proposed estimator variance can be expressed as 
\begin{equation}
    \label{eqn:mldlmc_est_variance}
    \Var{\mathcal{A}_{\textrm{MLMC}}} = \sum_{\ell=0}^L \frac{1}{M_{1,\ell}} \Var{\frac{1}{M_{2,\ell}} \sum_{j=1}^{M_{2,\ell}} \left(\Delta G_\ell\right)^{(1,j)}},
\end{equation}
where $(\Delta G_\ell)^{(i,j)} = (G_\ell - \mathcal{G}_{\ell-1})(\omega_{1:P_\ell}^{(\ell,i)},\bar{\omega}^{(\ell,j)})$. Using the law of total variance,
\begin{align}
    \text{Var}[\mathcal{A}_{\textrm{MLMC}}] &= \sum_{\ell=0}^L \frac{1}{M_{1,\ell}} \left(\underbrace{\text{Var}[\mathbb{E}[\Delta G_\ell \mid  \{\mu_\ell,\mu_{\ell-1}\}]]}_{=V_{1,\ell}} 
    + \frac{1}{M_{2,\ell}} \underbrace{\mathbb{E}[\text{Var}[\Delta G_\ell \mid \{\mu_\ell,\mu_{\ell-1}\}]]}_{V_{2,\ell}}\right) \nonumber\\
    \label{eqn:mldlmc_total_variance}
    &= \sum_{\ell=0}^L \left(\frac{V_{1,\ell}}{M_{1,\ell}} + \frac{V_{2,\ell}}{M_{1,\ell}M_{2,\ell}}\right) ,
\end{align}
where laws $\{\mu_\ell,\mu_{\ell-1}\}$ are coupled by the same sets of random variables $\omega_{1:P_\ell}^{(\ell,\cdot)}$ as described in the naïve \eqref{eqn:naive_sampler} and antithetic \eqref{eqn:antithetic_sampler} samplers. We make the following assumptions on $V_{1,\ell}$ and $V_{2,\ell}$.

\begin{assump}[{Multilevel DLMC} variance]
    \label{ass:mldlmc_var}
    \hspace{2mm} There exist constants $\tilde{w}>0$ and $\tilde{s}>0$ such that,
    \begin{align}
        V_{1,\ell} &= \text{Var}[\mathbb{E}[\Delta G_\ell \mid \{\mu_\ell,\mu_{\ell-1}\}]] \lesssim \tau^{-\Tilde{w} \ell} ,\\
        V_{2,\ell} &= \mathbb{E}[\text{Var}[\Delta G_\ell \mid \{\mu_\ell,\mu_{\ell-1}\}]] \lesssim \tau^{-\Tilde{s} \ell} \cdot
    \end{align}
\end{assump}
where constants $\Tilde{w}$ and $\Tilde{s}$ are the convergence rates for $V_{1,\ell}$ and $V_{2,\ell}$, respectively, with respect to $\ell$. Section~\ref{sec:results} numerically determines these constants for the Kuramoto model.

\subsubsection{Complexity analysis}

We can express the total {computational cost} of the proposed {multilevel DLMC} estimator using the cost of the DLMC estimator \eqref{eqn:dmvsde_dlmc_est},
\begin{equation}
    \label{eqn:mldlmc_work}
    \mathcal{W}[\mathcal{A}_{\textrm{MLMC}}] \lesssim \sum_{\ell=0}^L \left(M_{1,\ell} P_\ell^{1+\gamma_p} N_\ell^{\gamma_n} + M_{1,\ell} M_{2,\ell} P_\ell^{\gamma_p} N_\ell^{\gamma_n}\right) \cdot
\end{equation}
{$\gamma_p > 0$ and $\gamma_n > 0$ are the {computational complexity} rates of the empirical measure computation in the drift and diffusion coefficients and the time discretization scheme, respectively (see Appendix~\ref{app:work}).} Let some level $L$ satisfy the bias constraint \eqref{eqn:mldlmc_bias_constraint}. We aim to compute the optimal parameters $\{M_{1,\ell}$,$M_{2,\ell}\}_{\ell=0}^L$ satisfying \eqref{eqn:mldlmc_var_constraint}, which can be posed as the following optimization problem: 

\begin{empheq}[left=\empheqlbrace, right = \cdot]{equation}
    \label{eqn:mldlmc_optim}
    \begin{alignedat}{2}
        \arg \min_{\{M_{1,\ell},M_{2,\ell}\}_{\ell=0}^L} &  \sum_{\ell=0}^L \left(M_{1,\ell} P_\ell^{1+\gamma_p} N_\ell^{\gamma_n} + M_{1,\ell} M_{2,\ell} P_\ell^{\gamma_p} N_\ell^{\gamma_n}\right)\\ 
        \text{s.t. } & C_{\nu}^2\left(\sum_{\ell=0}^L \left(\frac{V_{1,\ell}}{M_{1,\ell}} + \frac{V_{2,\ell}}{M_{1,\ell}M_{2,\ell}}\right)\right) \approx (1-\theta)^2 \tol^2_\mathrm{r} \abs{\E{G}}^2 
    \end{alignedat}
\end{empheq}

{The arguments in \eqref{eqn:mldlmc_optim} in $\mathbb{R}^+$ that minimize the objective is found using the Lagrangian multiplier method $\forall \ell=0,\ldots,L$.}

\begin{align}
    \mathcal{M}_{1,\ell} &= \frac{C_\nu^2}{(1-\theta)^2 \tol_\mathrm{r}^2 \abs{\E{G}}^2} \frac{\sqrt{V_{1,\ell}}}{\sqrt{P_\ell^{1+\gamma_p} N_\ell^{\gamma_n}}} \left(\sum_{j=0}^L \sqrt{P_j^{\gamma_p} N_j^{\gamma_n}} (\sqrt{V_{1,j}P_j} + \sqrt{V_{2,j}})\right), \nonumber\\
    \label{eqn:mldlmc_optimal_samples}
    \tilde{\mathcal{M}}_{\ell} &= \mathcal{M}_{1,\ell} \mathcal{M}_{2,\ell} = \frac{C_\nu^2}{(1-\theta)^2 \tol_\mathrm{r}^2 \abs{\E{G}}^2} \frac{\sqrt{V_{2,\ell}}}{\sqrt{P_\ell^{\gamma_p} N_\ell^{\gamma_n}}} \left(\sum_{j=0}^L \sqrt{P_j^{\gamma_p} N_j^{\gamma_n}} (\sqrt{V_{1,j}P_j} + \sqrt{V_{2,j}})\right) \cdot
\end{align}

In practice, we only use natural numbers for $\{M_{1,\ell},M_{2,\ell}\}_{\ell=0}^L$. Hence, we use the following quasi-optimal solution to~\eqref{eqn:mldlmc_optim}:

\begin{equation}
	\label{eqn:mldlmc_actual_samples}
	M_{1,\ell} = \lceil \mathcal{M}_{1,\ell} \rceil, \quad M_{2,\ell} = \left\lceil  \frac{\tilde{\mathcal{M}}_{\ell}}{\lceil \mathcal{M}_{1,\ell} \rceil} \right\rceil \cdot
\end{equation}
{Note that $M_{1,\ell}$ required to satisfy $\tol_\mathrm{r}$ scales with the factor $\frac{1}{\abs{\E{G}}^2}$, implying that DLMC without importance sampling can quickly become computationally expensive for rare events.} We obtain the optimal {computational cost} for the proposed {multilevel DLMC} estimator using \eqref{eqn:mldlmc_actual_samples}. 

\begin{thm}[Optimal {multilevel DLMC} complexity]
    \label{th:mldlmc_complexity}
    \hspace{1mm} {Let $G_\ell$ be an approximation for the random variable $G$, for every $\ell \in \mathbb{N}$, and $G_\ell^{(i,j)} \equiv G_\ell(\omega_{1:P_\ell}^{(\ell,i)},\bar{\omega}^{(\ell,j)})$ be a sample of $G_\ell$. Consider the {multilevel DLMC} estimator \eqref{eqn:mldlmc_estimator} with $G_{-1}^{(i,j)} = 0$. Let Assumptions~\ref{ass:mldlmc_bias} and \ref{ass:mldlmc_var} hold. Let the constants $\tau$, $\Tilde{\alpha}$, $\Tilde{w}$, $\Tilde{s}$, $\gamma_p$, and $\gamma_n > 0$ from Assumptions~\ref{ass:mldlmc_bias} and \ref{ass:mldlmc_var} be such that $\tilde{\alpha} \geq \frac{1}{2} \min (\tilde{w}, 1+\tilde{s}, 1 + \gamma_p + \gamma_n)$.}
	Then, for any $\tol_\mathrm{r}<1/e$, there exists an optimal $L$ and sequences $\{M_{1,\ell}\}_{\ell=0}^L$ and $\{M_{2,\ell}\}_{\ell=0}^L$ such that 
    \begin{equation}
        \label{eqn:mldlmc_objective_v2}
        \mathbb{P}\left[\frac{\abs{\mathcal{A}_{\textrm{MLMC}}(L) - \E{G}}}{\abs{\E{G}}} \geq \tol_\mathrm{r}\right] \leq \nu ,
    \end{equation}
    and
    \begin{equation}
        \label{eqn:mldlmc_optimal_work}
        \mathcal{W}[\mathcal{A}_{\textrm{MLMC}}] \lesssim 
            \tol_\mathrm{r}^{-2-\max\left(0,\frac{1+\gamma_n+\gamma_p-\Tilde{w}}{\Tilde{\alpha}},\frac{\gamma_n + \gamma_p - \tilde{s}}{\tilde{\alpha}}\right)}(\log\tol_\mathrm{r}^{-1})^{2 \mathcal{J}},
    \end{equation}
    {
    where
    \begin{equation*}
    	\mathcal{J} = \begin{cases}
    	1, &\quad \text{if } \min(\tilde{w},1+\tilde{s})=1+\gamma_p+\gamma_n, \\
    	0, &\quad \text{else}.
    	\end{cases}
    \end{equation*}
    }
\end{thm}

\begin{proof}
    See Appendix~\ref{app:1}.
\end{proof}
{
\begin{rem}
\label{rem:optimal_parameters}
	\hspace{1mm} The complexity rate in \eqref{eqn:mldlmc_optimal_work} is independent of $\theta$ and $\tau$ or the dimension, $d$, of the MV-SDE. These only affect the associated constant in \eqref{eqn:mldlmc_optimal_work}. The optimization of the parameters $\theta$ and $\tau$ in the context of general {MLMC} methods has been investigated~\citep{optimal_mlmc,cont_mlmc}.
\end{rem}
}

\begin{rem}
	\hspace{1mm} In many applications, the variance and second moments of the level differences in the {multilevel DLMC} estimator are of the same order. In this case, we can demonstrate that $V_{1,\ell} \leq \E{\E{\Delta G^2_\ell \mid \{\mu_\ell,\mu_{\ell-1}\}}} \approx V_{2,\ell}$, implying $\tilde{w} \geq \tilde{s}$. 
\end{rem}
{
\begin{rem}
	\hspace{1mm} In the non-rare event context, $\tol_\mathrm{r}$ in~\eqref{eqn:mldlmc_objective_v2} and \eqref{eqn:mldlmc_optimal_work} is easily replaced by $\frac{\tol}{\abs{\E{G}}}$, where $\tol$ is the required absolute error tolerance.
\end{rem}
}
\begin{rem}[Kuramoto model]
    \label{rem:kuramoto}
    \hspace{1mm} $\gamma_p=1$ corresponds to a naïve empirical mean estimation method and $\gamma_n=1$ corresponds to the Euler--Maruyama scheme with a uniform time grid with time step $\frac{T}{N_\ell}$. Then, $\tilde{\alpha}=1$ due to the standard rates of weak convergence with respect to $P_\ell$~\citep{kolokoltsov2019mean,my_paper} and $N_\ell$ using the Euler--Maruyama scheme~\citep{sde_numerics}. We obtain better complexity than $\order{\tol_\mathrm{r}^{-4}}$ for the single-level DLMC estimator when $\tilde{s}+1 \geq \tilde{w}$ and $\tilde{w} > 1$. For this example with $G(x)=\cos(x)$, using the antithetic sampler ensures $\{\tilde{w},\tilde{s}\}=\{2,2\}$, leading to $\order{\tol_\mathrm{r}^{-3}}$ complexity. In contrast, the naïve sampler results in $\{\tilde{w},\tilde{s}\}=\{1,1\}$, leading to a complexity of $\order{\tol_\mathrm{r}^{-4}}$ (i.e., the same as the single-level DLMC estimator). Section~\ref{sec:results} presents these outcomes numerically. 
\end{rem}

Section~\ref{sec:is_mldlmc} devises an importance sampling scheme for the proposed {multilevel DLMC} estimator to address rare events associated with MV-SDEs.

\subsection{Importance Sampling Scheme for the {multilevel DLMC} Estimator for the Decoupled MV-SDE}
\label{sec:is_mldlmc}

We propose the following method to couple importance sampling with the proposed {multilevel DLMC} estimator. {We obtain one importance sampling control $\zeta$ off-line by solving the control PDE~\eqref{eqn:dmvsde_hjb_form3} derived in Section~\ref{sec:is_dlmc} using one realization of the stochastic particle system with a large number of particles $\Bar{P}$ and time steps $\Bar{N}$.} We apply the same control across all levels $\ell=0,\ldots, L$ in the proposed {multilevel DLMC} estimator \eqref{eqn:mldlmc_estimator}. Thus, we can rewrite the quantity of interest as 

\begin{equation}
    \label{eqn:mldlmc_is_estimator}
    \E{G_L} 
    = \sum_{\ell=0}^L \E{G_\ell - G_{\ell-1}}   
    = \sum_{\ell=0}^L \E{G^\zeta_\ell \mathbb{L}_\ell - G^\zeta_{\ell-1} \mathbb{L}_{\ell-1}},
\end{equation}

where
{
\begin{align}
    \label{eqn:mldlmc_is_qoi}
    G_\ell^\zeta &= G(\Bar{X}^{P_\ell|N_\ell}_\zeta(T)), \\
    \label{eqn:mldlmc_is_llhood}
    \mathbb{L}_\ell &= \prod_{n=0}^{N_\ell-1} \exp{-\frac{1}{2} \Delta t_\ell \norm{\zeta(t_{n,\ell},\Bar{X}^{P_\ell|N_\ell}_\zeta(t_{n,\ell}))}^2 -  \langle \Delta W (t_{n,\ell}), \zeta(t_{n,\ell},\Bar{X}^{P_\ell|N_\ell}_\zeta(t_{n,\ell})) \rangle} ;
\end{align}
}
{$\{\Bar{X}^{P_\ell|N_\ell}_\zeta(t_{n,\ell})\}_{n=0}^{N_\ell}$ is the time-discretized controlled, decoupled MV-SDE sample path at level $\ell$ (see Section~\ref{sec:dlmc_single}).} $\mathbb{L}_\ell$ is the likelihood factor at level $\ell$, and $\Delta t_\ell = \frac{T}{N_\ell}$ is the uniform time step of the Euler--Maruyama scheme for the decoupled MV-SDE at level $\ell$. {$\{\Delta W(t_{n,\ell})\}_{n=0}^{N_\ell} \sim \mathcal{N} (0, \sqrt{\Delta t_\ell} \mathbb{I}_d)$ are Wiener increments driving the dynamics of coarse and fine time-discretized paths of the decoupled MV-SDE at level $\ell$.} We define the proposed {multilevel DLMC} estimator with importance sampling as follows:

\begin{equation}
    \label{eqn:mldlmcis_estimator}
    \E{G_L} \approx \mathcal{A}^{\textrm{IS}}_{\textrm{MLMC}} = \sum_{\ell=0}^L \frac{1}{M_{1,\ell}} \sum_{i=1}^{M_{1,\ell}} \frac{1}{M_{2,\ell}} \sum_{j=1}^{M_{2,\ell}} (G^\mathrm{IS}_\ell - \mathcal{G}^\mathrm{IS}_{\ell-1})(\omega_{1:P_\ell}^{(\ell,i)},\bar{\omega}^{(\ell,j)}),
\end{equation} 
where $G^\mathrm{IS}_\ell(\omega_{1:P_\ell}^{(\ell,i)},\bar{\omega}^{(\ell,j)}) = G^\zeta_\ell \mathbb{L}_\ell(\omega_{1:P_\ell}^{(\ell,i)},\bar{\omega}^{(\ell,j)})$. 
{
\begin{rem}
\label{rem:work_dimension}
	\hspace{1mm} The complexity of the {multilevel DLMC} estimator with the above importance sampling scheme remains the same as in \eqref{eqn:mldlmc_optimal_work} because the optimal control problem \eqref{eqn:dmvsde_hjb_form3} is solved once and we do not include its cost in the complexity. 
\end{rem}

\begin{rem}
\label{rem:mlmc_variance}
\hspace{1mm} This is a natural extension to the importance sampling scheme previously developed for the single-level estimator~\citep{my_paper}. The optimal control $\zeta$ minimizes the single-level estimator variance of the conditional expectation $\E{G_\ell \mid \mu_\ell}$; thus, we expect a variance reduction for the level differences estimator. However, the optimal control $\zeta$ derived in Section~\ref{sec:optimal_is} minimizes $\Var{G_\ell}$ and not the variance of the MC estimator of $\E{\Delta G_\ell}$. Optimally, we must determine a control that minimizes $\Var{\Delta G_\ell}$ at each level $\ell$ of the {multilevel DLMC} estimator, which we leave for future work.
\end{rem}
}
Algorithm~\ref{alg:mldlmc_ld} in Appendix~\ref{app:2} implements the proposed importance sampling scheme in the level-difference estimator for the proposed {multilevel DLMC} method and can be easily modified for any other correlated sampler, such as the naïve sampler. Next, we build an adaptive {multilevel DLMC} algorithm that sequentially chooses parameters $L,\{M_{1,\ell},M_{2,\ell}\}_{\ell=0}^L$ satisfying constraints~\eqref{eqn:mldlmc_bias_constraint} and \eqref{eqn:mldlmc_stat_constraint}. The bias and variances $V_{1,\ell}$ and $V_{2,\ell}$ corresponding to level $\ell$ must be estimated cheaply and robustly to develop such an algorithm.

{
\subsection{Adaptive Multilevel Double Loop Monte Carlo Algorithm With Importance Sampling}
\label{sec:adaptive_mldlmcis}

\subsubsection{Estimating Bias at level $\ell$}

The bias for level $\ell$ can be approximated using Richardson extrapolation~\citep{mlmc_richardson_extra}:

\begin{equation}
	\label{eqn:mldlmc_bias_approx}
	\abs{\E{G} - \E{G_\ell}} \approx \left(1- \tau^{-\tilde{\alpha}} \right)^{-1} \abs{\E{G_{\ell+1} - G_\ell}} \cdot 
\end{equation}

Then, we use Algorithm~\ref{alg:mldlmc_ld} with at least $\underline{M}_1$ and $\underline{M}_2$ samples to obtain a DLMC estimation of the bias. To ensure robust bias estimates at all levels, we actually use $\hat{M}_1$ and $\hat{M}_2$ samples in Algorithm~\ref{alg:mldlmc_ld} defined as follows: 

\begin{align}
    \hat{M}_1 &= \max(M_{1,\ell},\underline{M}_1), \nonumber \\
    \label{eqn:bias_samples}
    \hat{M}_2 &= \max(M_{2,\ell},\underline{M}_2) \cdot
\end{align}

To ensure reliable bias estimates at higher levels ($\ell > 3$), we compare the DLMC bias estimator with the extrapolated bias from two previous levels using Assumption~\ref{ass:mldlmc_bias}:

\begin{equation}
\label{eqn:mldlmc_robust_bias}
	\abs{\E{G} - \E{G_\ell}} \approx \max \left( \frac{\abs{\E{\Delta G_{\ell+1}}}}{1- \tau^{-\tilde{\alpha}}}, \frac{\abs{\E{G} - \E{G_{\ell-1}}}}{\tau^{\tilde{\alpha}}}, \frac{\abs{\E{G} - \E{G_{\ell-2}}}}{\tau^{2\tilde{\alpha}}} \right), \quad \ell > 2 \cdot
\end{equation}

\subsubsection{Estimating $V_{1,\ell},V_{2,\ell}$ at Level $\ell$}

To compute the optimal number of samples required to satisfy the statistical error constraint~\eqref{eqn:mldlmc_var_constraint} using~\eqref{eqn:mldlmc_actual_samples}, we require cheap and robust empirical estimates of the variance terms $V_{1,\ell}$ and  $V_{2,\ell}$ for each level $\ell$. For this, we apply the DLMC algorithm (Algorithm~\ref{alg:mldlmc_est_variance} in Appendix~\ref{app:3}) with appropriately chosen values of $\tilde{M}_1$ and $\tilde{M}_2$. Algorithm~\ref{alg:mldlmc_est_variance} could become computationally expensive at higher levels. We exploit Assumption~\ref{ass:mldlmc_var} to avoid this overload and extrapolate variances for higher levels. For levels $\ell > 3$,

\begin{align}
\label{eqn:mldlmc_robust_var}
	V_{1,\ell} &= \max \left( \frac{V_{1,\ell-1}}{\tau^{\tilde{w}}}, \frac{V_{1,\ell-2}}{\tau^{2\tilde{w}}} \right), \\
	V_{2,\ell} &= \max \left( \frac{V_{2,\ell-1}}{\tau^{\tilde{s}}}, \frac{V_{2,\ell-2}}{\tau^{2\tilde{s}}} \right) \cdot \nonumber
\end{align}

\subsubsection{Relative Error Control}

To control the relative bias and statistical errors, we require a heuristic estimate of the quantity of interest $\E{G}$ itself. This estimate is continuously updated at each level $L$. At level $L=0$, we use the DLMC algorithm (Algorithm~\ref{alg:mldlmc_ld}) with $\bar{M}_1$ and $\bar{M}_2$ samples to obtain an initial robust but cheap estimate of $\E{G}$. For subsequent levels, we apply the {multilevel DLMC} estimator~\eqref{eqn:mldlmc_estimator} with optimal values for $\{M_{1,\ell},M_{2,\ell}\}_{\ell=0}^L$ to update the estimate. Putting all this together, we propose the adaptive {multilevel DLMC} algorithm (Algorithm~\ref{alg:mldlmc_is_adaptive}) for rare-event observables in the MV-SDE context. The IS control $\zeta$ in Algorithm~\ref{alg:mldlmc_is_adaptive} is obtained off-line by generating one realization of the empirical law $\mu^{\bar{P} \mid \bar{N}}$ with large $\bar{P}$ and $\bar{N}$ and then numerically solving control Equation~\eqref{eqn:dmvsde_hjb_form3} given $\mu^{\bar{P} \mid \bar{N}}$.

\begin{algorithm}[H] 
	\label{alg:mldlmc_is_adaptive}
	\caption{Adaptive {multilevel DLMC} algorithm with importance sampling}
	\SetAlgoLined
	\textbf{Input: } $P_0,N_0,\tol_{\mathrm{r}},\zeta(\cdot,\cdot)$,$\{\bar{M}_1,\bar{M}_2\}$,$\{\tilde{M}_1,\tilde{M}_2\}$,$\{\underline{M}_1,\underline{M}_2\}$; \\
	Estimate $\bar{G} = \E{G_0}$ with $P_0,N_0, \bar{M}_1,\bar{M}_2,\zeta(\cdot,\cdot)$ using \textbf{Algorithm \ref{alg:mldlmc_ld}}; \\
	Estimate and store $V_{1,0}$,$V_{2,0}$ with $P_0,N_0, \bar{M}_1,\bar{M}_2,\zeta(\cdot,\cdot)$ using \textbf{Algorithm \ref{alg:mldlmc_est_variance}};\\
	Set $L = 1$; \\
	\While{Bias $> \theta \tol_{\mathrm{r}} \bar{G}$}{
		$P_L=P_0 2^L,\quad N_L = N_0 2^{L}$; \\
		Estimate and store $V_{1,L},V_{2,L}$ with $P_L,N_L,\tilde{M}_1,\tilde{M}_2,\zeta(\cdot,\cdot)$ using \textbf{Algorithm \ref{alg:mldlmc_est_variance}}; \\
		Compute optimal $\{M_{1,\ell},M_{2,\ell}\}_{\ell=0}^L$ using \eqref{eqn:mldlmc_actual_samples}; \\
		Estimate bias using \eqref{eqn:mldlmc_bias_approx} with $P_{L+1},N_{L+1},\Hat{M}_1,\Hat{M}_2,\zeta(\cdot,\cdot)$ using \eqref{eqn:bias_samples} and \textbf{Algorithm \ref{alg:mldlmc_ld}}; \\
		Reevaluate $\bar{G}=\mathbb{E}[G_0] + \sum_{\ell=1}^L \mathbb{E}[\Delta G_\ell]$ with $\{M_{1,\ell},M_{2,\ell}\}_{\ell=0}^L,\zeta(\cdot,\cdot)$ using \textbf{Algorithm \ref{alg:mldlmc_ld}} for each $\ell$; \\
		$L \longleftarrow L+1$;
	}
	$\mathcal{A}_{\textrm{MLMC}} = \bar{G}$.
\end{algorithm}
}

\section{Numerical Results}
\label{sec:results}

This section provides numerical evidence for the assumptions and rates of {computational complexity} derived in Section~\ref{sec:mldlmc}. The results outlined below focus on the Kuramoto model (see Section~\ref{sec:kuramoto}) with the following settings: $\sigma=0.4$, $T=1$, $(x_0)_p \sim \mathcal{N}(0,0.2)$, and {$\xi_p \sim \mathcal{U}(-0.2,0.2)$} for all $p = 1,\ldots,P$. We set the parameters as follows: $\tau=2$, $\theta=0.5$, and {$\nu=0.05$}. The complexity rates ($\gamma_p=1$ and $\gamma_n=1$) are explained in Remark~\ref{rem:kuramoto}. Moreover, we set the hierarchies for the {multilevel DLMC} estimator as

\begin{equation}
    P_\ell = 5 \times 2^\ell, \quad N_\ell = 4 \times 2^\ell \cdot
\end{equation}

We implement the proposed {multilevel DLMC} method for nonrare and rare-event observables and investigate the {computational complexity} compared with the single-level DLMC estimator. 

\subsection{Objective Function $G(x)=\cos(x)$}

First, we numerically verify Assumptions~\ref{ass:mldlmc_bias} and \ref{ass:mldlmc_var} for the smooth, nonrare observable $G(x) = \cos(x)$ without importance sampling to determine constants $\tilde{\alpha},\tilde{w},\tilde{s}$ for the Kuramoto model. Figure~\ref{fig:mldlmc_constants} presents the estimated bias \eqref{eqn:mldlmc_bias_approx} using Algorithm~\ref{alg:mldlmc_ld}, and $V_{1,\ell}$ and $V_{2,\ell}$ with respect to $\ell$ using Algorithm~\ref{alg:mldlmc_est_variance}. Thus, Assumptions~\ref{ass:mldlmc_bias} and \ref{ass:mldlmc_var} are verified with $\tilde{\alpha}=1$ and $\{\tilde{w},\tilde{s}\}=\{1,1\}$ for the naïve sampler \eqref{eqn:naive_sampler} and with $\{\tilde{w},\tilde{s}\}=\{2,2\}$ for the antithetic sampler \eqref{eqn:antithetic_sampler}. Improved variance convergence rates for the antithetic sampler imply a complexity of $\order{\tol^{-3}}$ for the proposed {multilevel DLMC} estimator, compared with $\order{\tol^{-4}}$ for the naïve sampler (see Theorem~\ref{th:mldlmc_complexity}) to achieve a prescribed absolute error tolerance $\tol$. Thus, we use the antithetic sampler in the proposed adaptive algorithm (Algorithm~\ref{alg:mldlmc_is_adaptive}).

\begin{figure}
    \centering
    \begin{subfigure}[b]{0.45\textwidth}
        \centering
        \includegraphics[width=\textwidth]{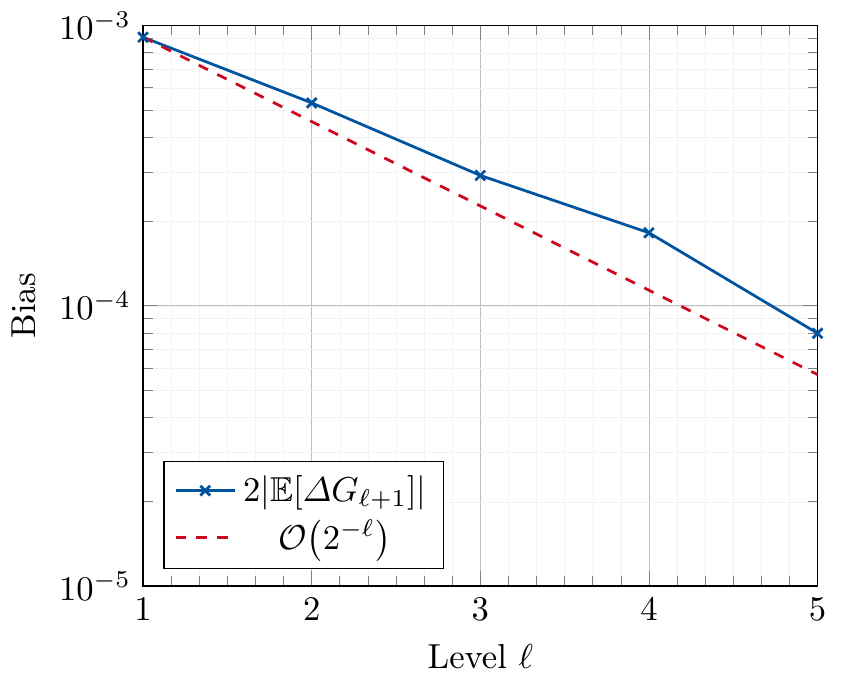}
        \caption{Assumption~\ref{ass:mldlmc_bias}: Double loop Monte Carlo bias estimator using Algorithm~\ref{alg:mldlmc_ld} with inputs $M_1=10^3$ and $M_2=10^3$ with respect to level $\ell$.}
        \label{fig:mldlmc_alpha}
    \end{subfigure}
    \hfill
    \begin{subfigure}[b]{0.45\textwidth}
        \centering
        \includegraphics[width=\textwidth]{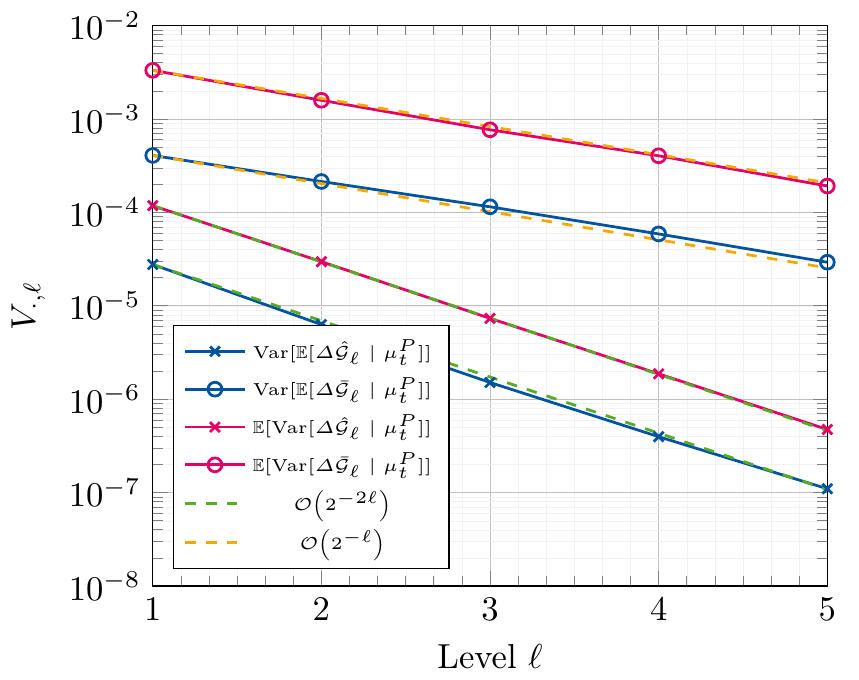}
        \caption{Assumption~\ref{ass:mldlmc_var}: Double loop Monte Carlo estimates for $V_{1,\ell}$ and $V_{2,\ell}$ using Algorithm~\ref{alg:mldlmc_est_variance} with inputs $M_1=10^2$ and $M_2=10^4$ with respect to level $\ell$.}
        \label{fig:mldlmc_c2l}
    \end{subfigure}
    \caption{Convergence rates of level differences using antithetic ($\hat{\mathcal{G}}$) and naïve ($\bar{\mathcal{G}}$) samplers for the Kuramoto model with $G(x)=\cos{x}$.}
    \label{fig:mldlmc_constants}
\end{figure}


\begin{figure}[t]
    \centering
    \begin{subfigure}[b]{0.45\textwidth}
        \centering
        \includegraphics[width=\textwidth]{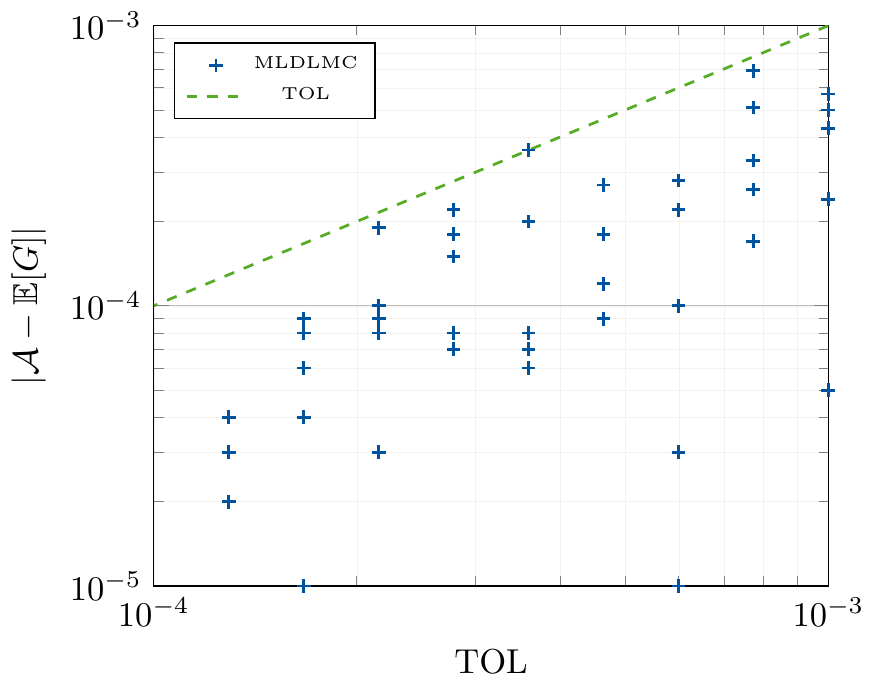}
        \caption{Global estimator error with respect to tolerance $\tol$}
        \label{fig:mlmc_globalerr}
    \end{subfigure}
    \hfill
	\begin{subfigure}[b]{0.45\textwidth}
        \centering
        \includegraphics[width=\textwidth]{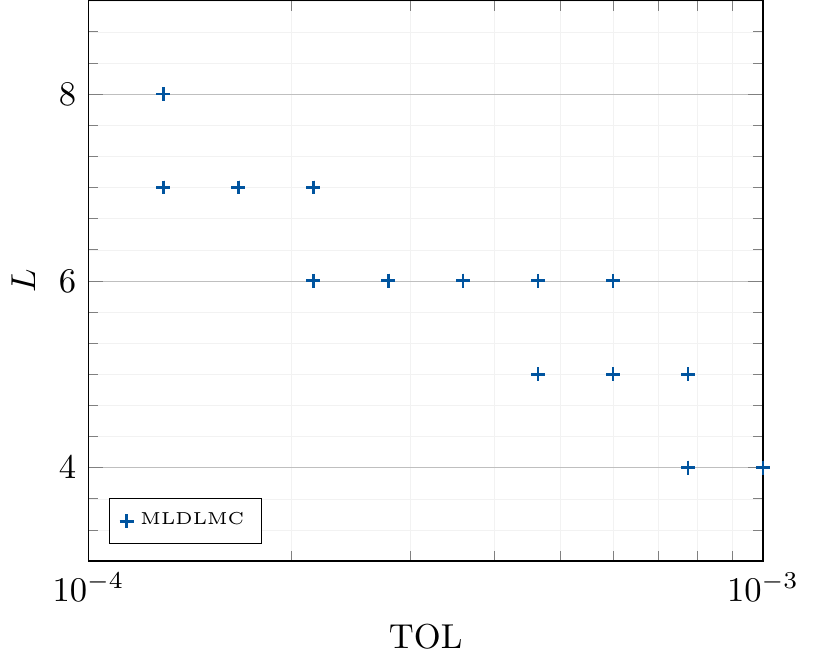}
        \caption{Estimator level $L$ required to satisfy the tolerance $\tol$.}
        \label{fig:mlmc_level}
    \end{subfigure}
    \hfill
    \begin{subfigure}[b]{0.45\textwidth}
        \centering
        \includegraphics[width=\textwidth]{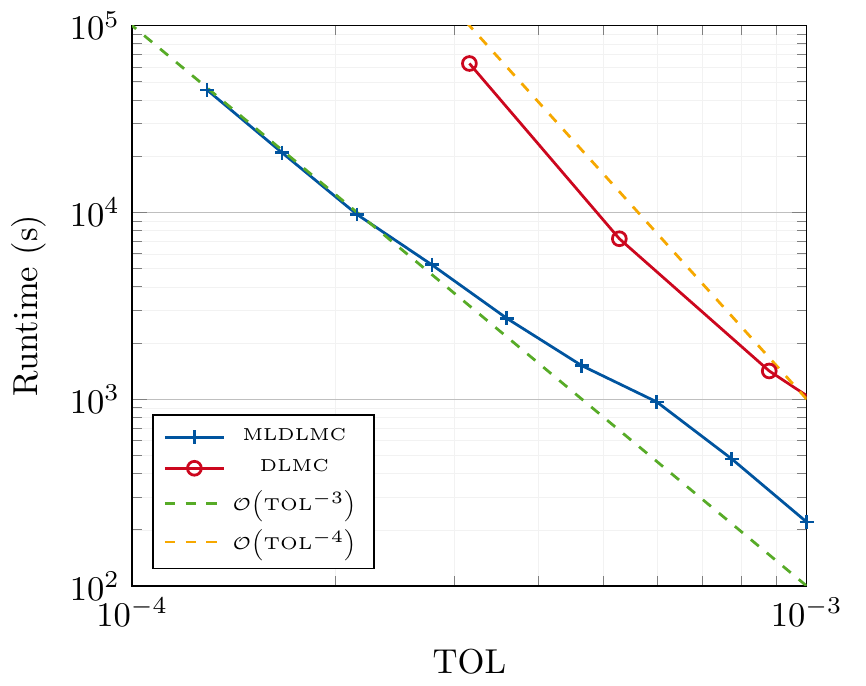}
        \caption{Average computational runtime with respect to the tolerance $\tol$.}
        \label{fig:mlmc_runtime}
    \end{subfigure}
    \hfill
    \begin{subfigure}[b]{0.45\textwidth}
        \centering
        \includegraphics[width=\textwidth]{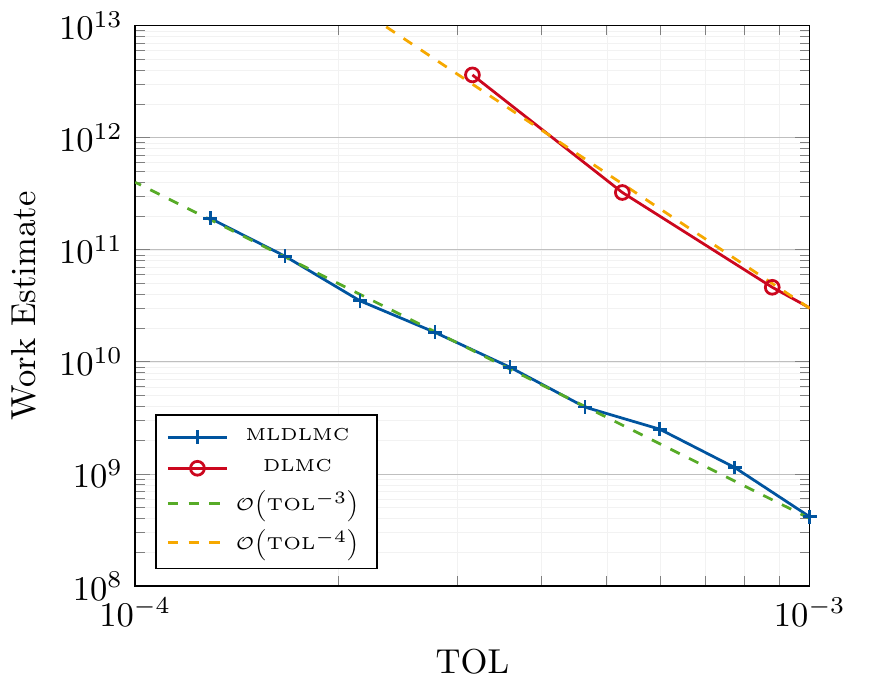}
        \caption{Average {computational cost} estimate with respect to the  tolerance $\tol$}
        \label{fig:mlmc_work}
    \end{subfigure}
    \caption{Algorithm~\ref{alg:mldlmc_is_adaptive} applied to the Kuramoto model for $G(x) = \cos{x}$. (MLDLMC: multilevel double loop Monte Carlo) . }
    \label{fig:mldlmc}
\end{figure}

Next, we implement the proposed {multilevel DLMC} algorithm (Algorithm~\ref{alg:mldlmc_is_adaptive}) with inputs $\{\bar{M}_1,\bar{M}_2\} = \{1000,100\}$, $\{\tilde{M}_1,\tilde{M}_2\} = \{25,1000\}$, and $\{\underline{M}_1,\underline{M}_2\} = \{100,50\}$ for this observable {with tuned parameters .} In this case, importance sampling is not required because this is not a rare-event observable, i.e. we set $\zeta(t,x) = 0,\forall (t,x) \in [0, T] \cross \mathbb{R}^d$. Figure~\ref{fig:mldlmc} depicts the results of Algorithm~\ref{alg:mldlmc_is_adaptive} in this setting to numerically verify the complexity rates obtained from Theorem~\ref{th:mldlmc_complexity}. Figure~\ref{fig:mlmc_globalerr} illustrates the exact {multilevel DLMC} estimator error, estimated using a reference {multilevel DLMC} approximation with $\tol = 10^{-4}$, for separate runs of Algorithm~\ref{alg:mldlmc_is_adaptive} for different prescribed absolute error tolerances $\tol$. The adaptive {multilevel DLMC} algorithm successfully satisfies the error constraint \eqref{eqn:mldlmc_objective_v2} at the 95\% confidence level (corresponding to $\nu=0.05$). {Figure~\ref{fig:mlmc_level} presents the number of levels $L$ required to satisfy the bias constraint in~\eqref{eqn:mldlmc_bias_constraint} for each of the separate runs of Algorithm~\ref{alg:mldlmc_is_adaptive} for different prescribed error tolerances $\tol$. According to Theorem~\ref{th:mldlmc_complexity}, the number of levels should increase by $\order{\log (\tol^{-1})}$.} Figure~\ref{fig:mlmc_runtime} displays the average computational runtime for Algorithm~\ref{alg:mldlmc_is_adaptive} for various error tolerances. The runtimes in Figure~\ref{fig:mlmc_runtime} include the cost of estimating the bias, $V_{1,\ell}$ and $V_{2,\ell}$, in Algorithm~\ref{alg:mldlmc_is_adaptive}. The runtimes for sufficiently small tolerances follow the predicted $\order{\tol^{-3}}$ rate from Theorem~\ref{th:mldlmc_complexity}. Figure~\ref{fig:mlmc_work} indicates the average estimated {computational cost} of the {multilevel DLMC} estimator for various $\tol$ values. {This estimated {computational cost} is computed using~\eqref{eqn:mldlmc_work}:

\begin{equation}
\label{eqn:mldlmc_cost}
	\text{{Computational cost}}[\mathcal{A}_{\mathrm{MLMC}}] \approx \sum_{\ell=0}^L \left(M_{1,\ell} P_\ell^{1+\gamma_p} N_\ell^{\gamma_n} + M_{1,\ell} M_{2,\ell} P_\ell^{\gamma_p} N_\ell^{\gamma_n}\right) \cdot
\end{equation}
}
Figures \ref{fig:mlmc_runtime} and \ref{fig:mlmc_work} verify that the proposed {multilevel DLMC} estimator with the antithetic sampler outperforms the single-level DLMC estimator, achieving one order of complexity reduction from $\order{\tol^{-4}}$ to $\order{\tol^{-3}}$.

\subsection{Rare-event Objective Function}
{
To test the importance sampling scheme, we consider the Kuramoto model with the following Lipschitz rare-event observable $G(x) = \Psi(x-K)$ for a sufficiently large threshold $K \in \mathbb{R}$, where

\begin{equation}
    \label{eqn:psi_defn}
    \Psi(x) = \begin{cases}
        0 &, \quad x<-0.5 \\
        0.5 + x &, \quad -0.5<x<0.5 \\
        1 &, \quad x > 0.5
    \end{cases}
    \cdot
\end{equation}
}
We use the importance sampling scheme introduced in Section~\ref{sec:is_mldlmc} with importance sampling control $\zeta$ obtained by solving \eqref{eqn:dmvsde_hjb_form3} numerically using finite differences and linear interpolation throughout the domain. First, we verify the variance reduction in the level-difference estimators using this $\zeta$ in two numerical experiments, whose results are depicted in Figure~\ref{fig:test_mlmcis}. 

\begin{figure}[t]
    \centering
    \begin{subfigure}[b]{0.45\textwidth}
        \centering
        \includegraphics[width=\textwidth]{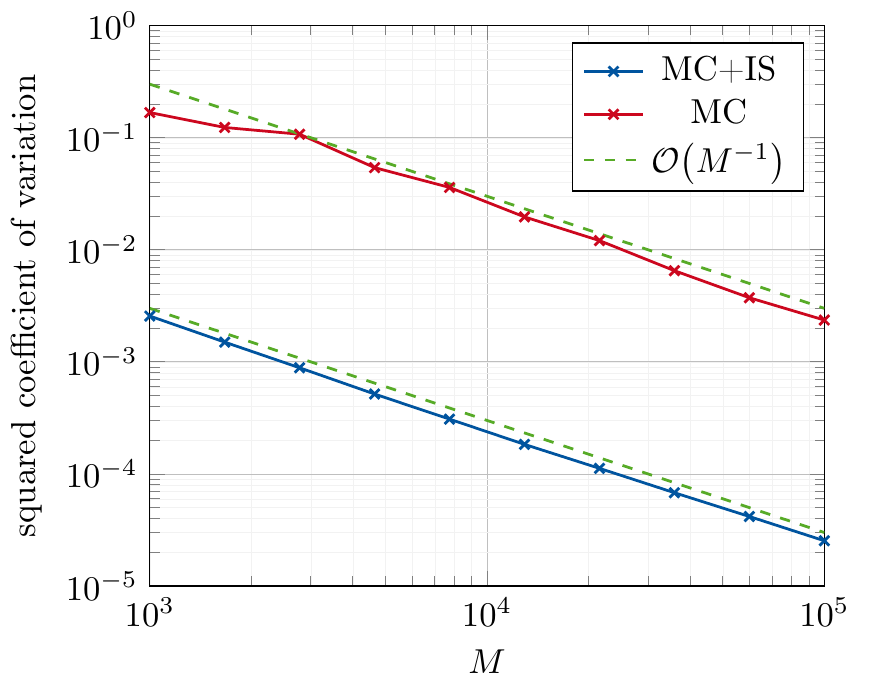}
        \caption{Experiment 1: Squared coefficient of variation of the double loop Monte Carlo estimator for $\E{\Delta G_3 \mid \mu^{\Bar{P}|\Bar{N}}}$ with and without importance sampling with respect to the number of sample paths $M$.}
        \label{fig:test1_mlmcis}
    \end{subfigure}
    \hfill
    \begin{subfigure}[b]{0.45\textwidth}
        \centering
        \includegraphics[width=\textwidth]{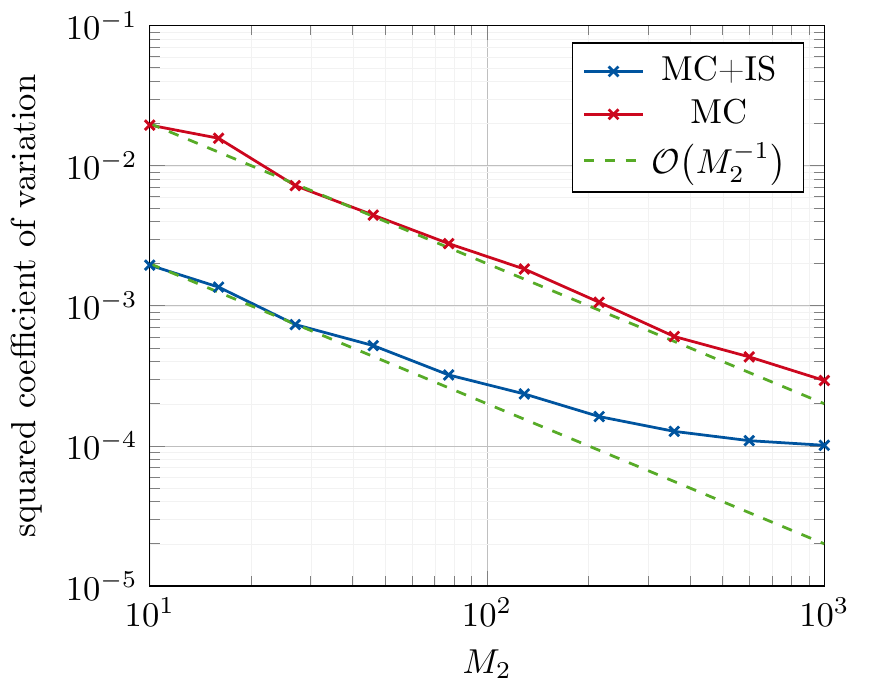}
        \caption{Experiment 2: Squared coefficient of variation of the double loop Monte Carlo estimator for $\E{\Delta G_3}$ with and without importance sampling with respect to the number of sample paths in inner loop $M_2$.}
        \label{fig:test2_mlmcis}
    \end{subfigure}
    \caption{Numerical experiments verifying variance reduction in the double loop Monte Carlo estimator for level differences using importance sampling on the Kuramoto model for $G(x)=\Psi(x-K)$ with $K=2.5$.}
    \label{fig:test_mlmcis}
\end{figure}

{
Note that $\zeta$ is the optimal importance sampling control for the decoupled MV-SDE \eqref{eqn:decoupled_mvsde} and not the particle system \eqref{eqn:strong_approx_mvsde}. With this scheme, we reduce the variance of the inner expectation in \eqref{eqn:total_exp_mvsde}. Consequently, we assess the variance reduction in the MC estimator of the inner expectation in the first experiment. Experiment~1 investigates variance reduction on the DLMC estimator of $\E{\Delta G_\ell \mid \mu^{\Bar{P}|\Bar{N}}}$. We run DLMC Algorithm~\ref{alg:mldlmc_ld} with $M_1 = 1$ and $\ell = 3$ for different values of $M_2$. To generate Figure~\ref{fig:test1_mlmcis}, we obtained the importance sampling control $\zeta$ using one realization of the empirical law $\mu^{\bar{P}|\bar{N}}$ with $\bar{P}=200$ particles and $\bar{N}=100$ time steps.} Figure~\ref{fig:test1_mlmcis} compares squared coefficients of variation for the DLMC estimator with and without importance sampling versus the number of sample paths $M$. The results verify that the importance sampling scheme reduces the squared coefficient of variation approximately 100-fold.
{
In Experiment 2, we verify the variance reduction for the DLMC estimator of $\E{\Delta G_\ell}$ with importance sampling using DLMC Algorithm~\ref{alg:mldlmc_ld} with $\ell=3$, $M_1=10^3$ for different values of $M_2$. To generate Figure~\ref{fig:test2_mlmcis}, we obtained importance sampling control $\zeta$ by solving control Equation~\eqref{eqn:dmvsde_hjb_form3} off-line using an independent realization of the empirical law $\mu^{\bar{P}|\bar{N}}$ with $\bar{P}=1000$ particles and $\bar{N}=100$ time steps. The results verify a significantly reduced squared coefficient of variation (approximately one order of magnitude) with importance sampling. The estimator variance of $\E{\Delta G_\ell}$ is given by $\frac{V_{1,\ell}}{M_{1,\ell}} + \frac{V_{2,\ell}}{M_{1,\ell} M_{2,\ell}}$ \eqref{eqn:mldlmc_est_variance}. We notice convergence of the squared coefficient of variation as the second term vanishes with a large $M_2$ value.}

\begin{figure}[t]
    \centering
    \begin{subfigure}[b]{0.45\textwidth}
        \centering
        \includegraphics[width=\textwidth]{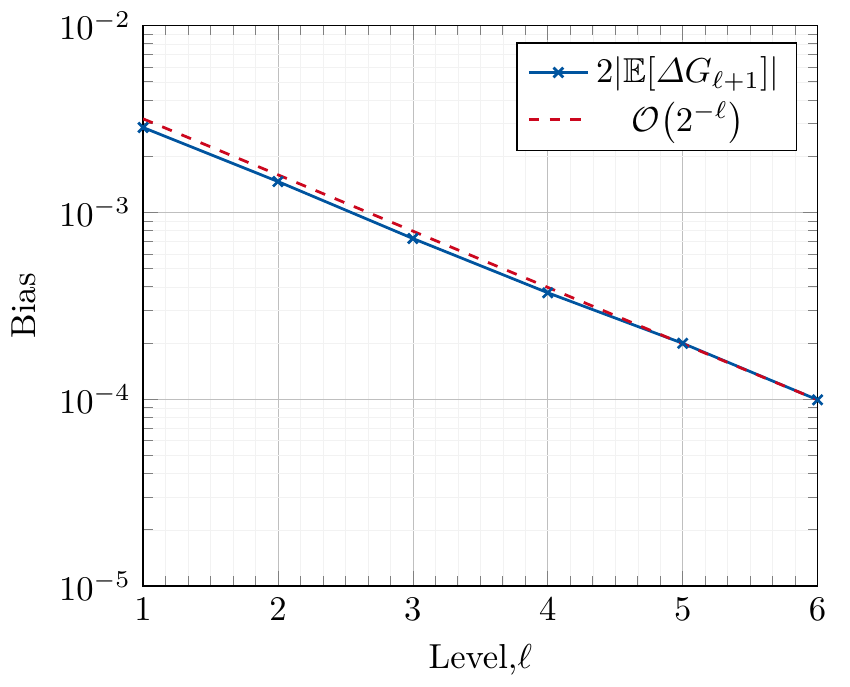}
        \caption{Assumption~\ref{ass:mldlmc_bias} for $G(x)=\Psi(x-K)$: Double loop Monte Carlo bias estimator using Algorithm~\ref{alg:mldlmc_ld} with inputs $M_1=10^2$ and $M_2=10^4$ with respect to $\ell$.}
        \label{fig:mldlmcis_bias}
    \end{subfigure}
    \hfill
    \begin{subfigure}[b]{0.45\textwidth}
        \centering
        \includegraphics[width=\textwidth]{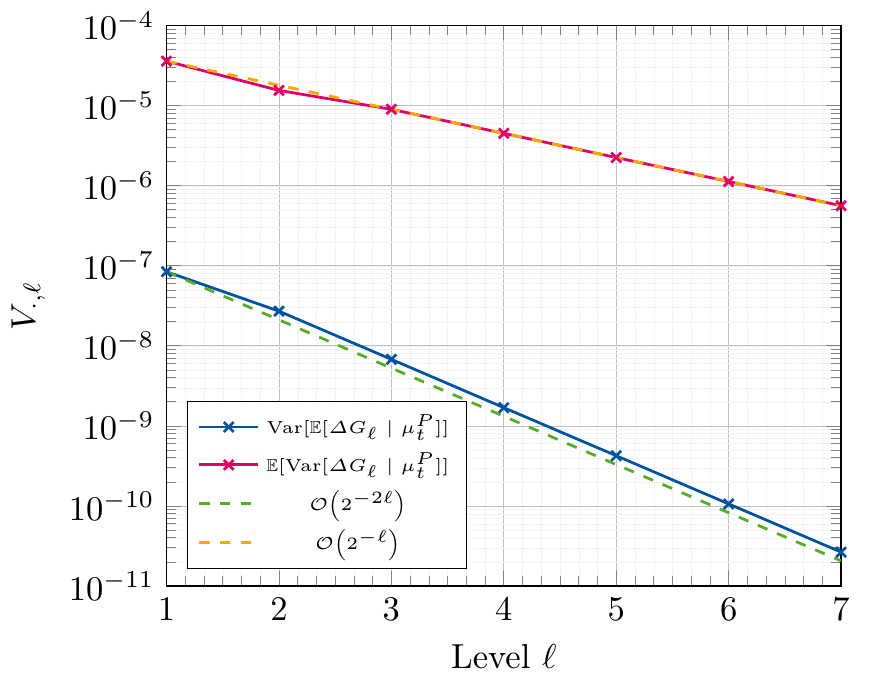}
        \caption{Assumption~\ref{ass:mldlmc_var} for $G(x)=\Psi(x-K)$: Double loop Monte Carlo estimator for $V_{1,\ell},V_{2,\ell}$ using Algorithm~\ref{alg:mldlmc_est_variance} with inputs $M_1=10^2$ and $M_2=10^4$ with respect to $\ell$.}
        \label{fig:mldlmcis_v1l}
    \end{subfigure}
    \caption{Convergence rates of level differences using the antithetic estimator ($\hat{\mathcal{G}}$) for the Kuramoto model with $G(x)=\Psi(x-K)$.}
    \label{fig:mldlmcis_constants}
\end{figure}
{
We numerically verify Assumptions~\ref{ass:mldlmc_bias} and \ref{ass:mldlmc_var} for this rare-event observable to determine constants $\tilde{\alpha}$, $\tilde{w}$, and $\tilde{s}$ for the Kuramoto model with importance sampling.} Figure~\ref{fig:mldlmcis_constants} verifies Assumptions~\ref{ass:mldlmc_bias} and \ref{ass:mldlmc_var} with $\tilde{\alpha}=1$ and $\{\tilde{w},\tilde{s}\}=\{2,1\}$ for the antithetic sampler \eqref{eqn:antithetic_sampler}. These rates imply a complexity of $\order{\tol_\mathrm{r}^{-3}}$ for the proposed {multilevel DLMC} estimator with importance sampling according to Theorem~\ref{th:mldlmc_complexity}.

\begin{figure}[t]
    \centering
    \begin{subfigure}[b]{0.45\textwidth}
        \centering
        \includegraphics[width=\textwidth]{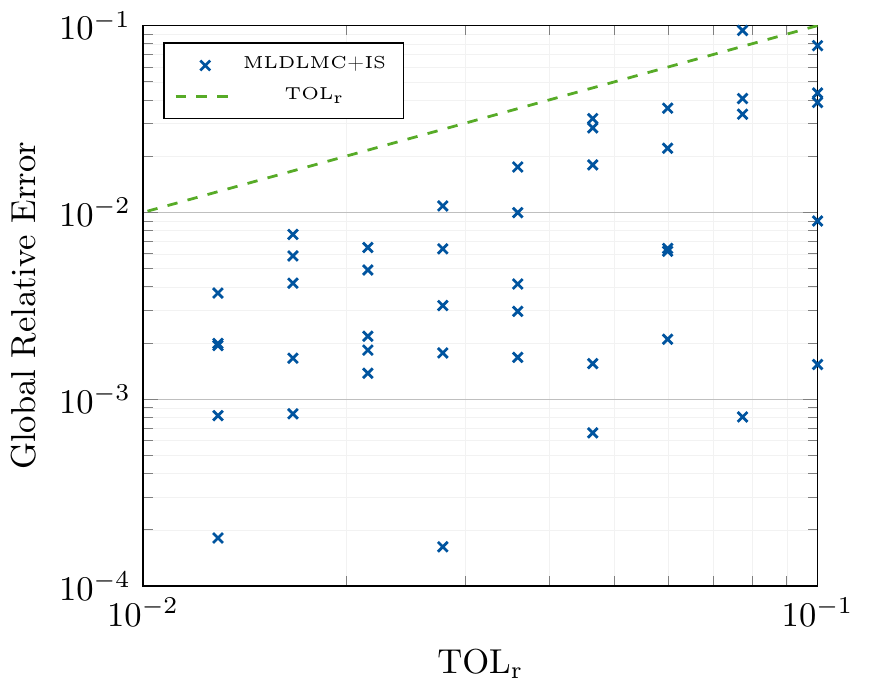}
        \caption{Relative error of the estimator with respect to the relative tolerance $\tol_{\mathrm{r}}$ for $K=2.5$.}
        \label{fig:mlmcis_globalerr}
    \end{subfigure}
    \hfill
    \begin{subfigure}[b]{0.45\textwidth}
        \centering
        \includegraphics[width=\textwidth]{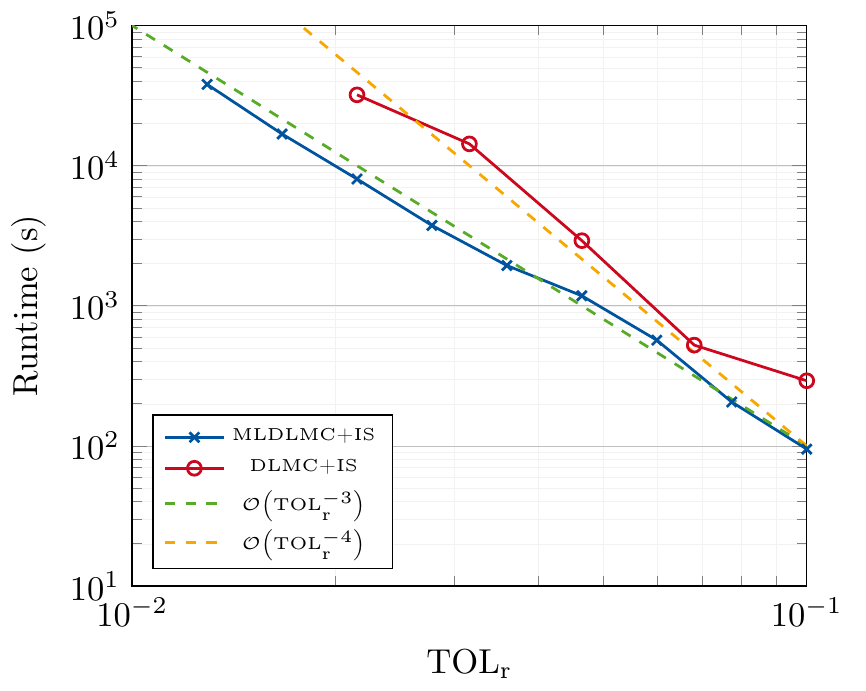}
        \caption{Average computational runtime with respect to the relative tolerance $\tol_{\mathrm{r}}$ for $K=2.5$.}
        \label{fig:mlmcis_runtime}
    \end{subfigure}
    \hfill
    \begin{subfigure}[b]{0.45\textwidth}
        \centering
        \includegraphics[width=\textwidth]{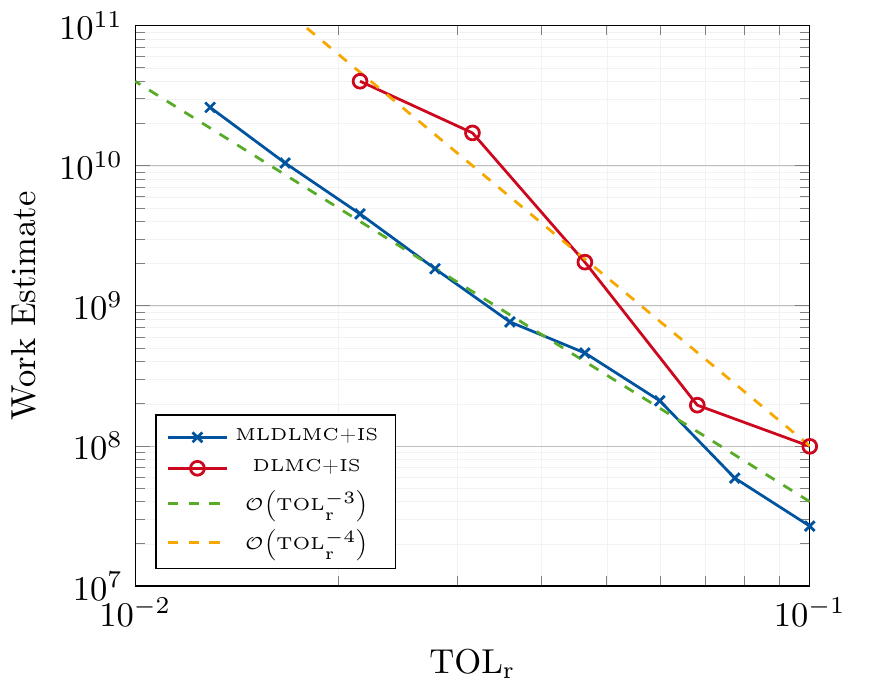}
        \caption{Average {computational cost} estimate with respect to the relative tolerance $\tol_{\mathrm{r}}$ for $K=2.5$.}
        \label{fig:mlmcis_work}
    \end{subfigure}
	\hfill
    \begin{subfigure}[b]{0.45\textwidth}
        \centering
        \includegraphics[width=\textwidth]{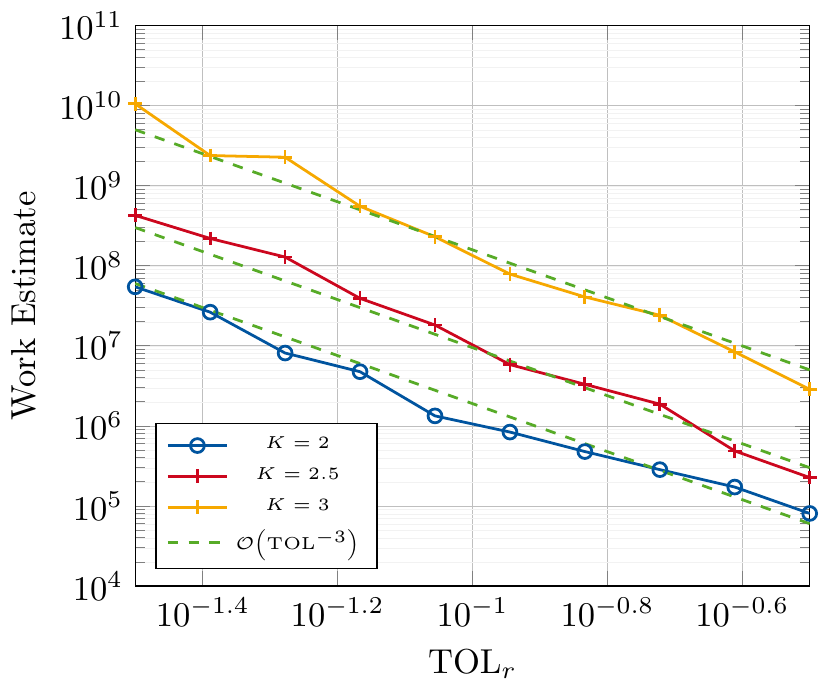}
        \caption{Average {computational cost} estimate with respect to the  relative tolerance $\tol_{\mathrm{r}}$ for different threshold ($K$) values.}
        \label{fig:mlmcis_work_k}
    \end{subfigure}
    \caption{Algorithm~\ref{alg:mldlmc_is_adaptive} applied to the Kuramoto model for $G(x) = \Psi(x-K)$.}
    \label{fig:mldlmc_is}
\end{figure}

We implement the proposed adaptive {multilevel DLMC} algorithm  (Algorithm~\ref{alg:mldlmc_is_adaptive}) {with inputs $\{\bar{M}_1,\bar{M}_2\} = \{1000,100\}$, $\{\tilde{M}_1,\tilde{M}_2\} = \{25,100\}$, and $\{\underline{M}_1,\underline{M}_2\} = \{100,50\}$} using importance sampling for the 1D Kuramoto model. We use $\bar{P}=1000$ particles and $\bar{N}=100$ time steps to independently obtain one empirical realization of $\mu^{\bar{P}|\bar{N}}$ to compute importance sampling control $\zeta$. Figure~\ref{fig:mldlmc_is} presents the results of Algorithm~\ref{alg:mldlmc_is_adaptive} in this setting and numerically verifies the complexity rates obtained in Theorem~\ref{th:mldlmc_complexity}. {We observe that $K=2.5$ corresponds to an expected value $\approx 3.2 \times 10^{-3}$, $K=2$ corresponds to an expected value of $\approx 2.3 \times 10^{-2}$, and $K=3$ corresponds to an expected value of $\approx 3.1 \times 10^{-4}$, all sufficiently rare events. }

Figures~\ref{fig:mlmcis_globalerr}, \ref{fig:mlmcis_runtime}, and \ref{fig:mlmcis_work} correspond to $K=2.5$. Figure~\ref{fig:mlmcis_globalerr} presents exact relative error for the proposed {multilevel DLMC} estimator, estimated using a reference {multilevel DLMC} approximation with $\tol_{\mathrm{r}} = 1\%$, for multiple runs of Algorithm~\ref{alg:mldlmc_is_adaptive} with various prescribed relative error tolerances. {Figure~\ref{fig:mlmcis_globalerr} verifies that the {multilevel DLMC} estimator with importance sampling satisfies error constraints with 95\% confidence (corresponding to $\nu = 0.05$).} The runtimes in Figure~\ref{fig:mlmcis_runtime} include the estimation time of the bias, $V_{1,\ell}$ and $V_{2,\ell}$, in adaptive Algorithm~\ref{alg:mldlmc_is_adaptive}. Figure~\ref{fig:mlmcis_runtime} confirms that the average computational runtime closely follows the predicted theoretical rate $\order{\tol_{\mathrm{r}}^{-3}}$ for the entire range of relative tolerances. Figure~\ref{fig:mlmcis_work} depicts the average {computational cost} estimate for the {multilevel DLMC} estimator over the prescribed $\tol_{\mathrm{r}}$ values. {The cost estimate is computed using \eqref{eqn:mldlmc_cost}.} Figures \ref{fig:mlmcis_runtime} and \ref{fig:mlmcis_work} both display a complexity of $\order{\tol_{\mathrm{r}}^{-3}}$ for the {multilevel DLMC} estimator with importance sampling and the antithetic sampler, achieving one order complexity reduction compared with the single-level DLMC estimator with the same importance sampling scheme. {Figure \ref{fig:mlmcis_work_k} plots the average work estimate for various threshold values $K$ over a range of relative error tolerances. This verifies that the complexity rates are independent of parameter $K$.}


\section{Conclusion}
\label{sec:conclusion}

Under certain assumptions that can be numerically verified, we have theoretically and numerically demonstrated the improvement of {multilevel DLMC} compared with the single-level DLMC estimator used to approximate rare-event quantities of interest expressed as an expectation of a Lipschitz observable of the solution to stochastic particle systems in the mean-field limit. We used the importance sampling scheme introduced in \citep{my_paper} for all level-difference estimators in the proposed {multilevel DLMC} estimator and verified substantial variance reduction numerically. The proposed novel {multilevel DLMC} estimator achieved a complexity of $\order{\tol_{\mathrm{r}}^{-3}}$ for the treated example, one order less than the single-level DLMC estimator for the prescribed relative error tolerance $\tol_{\mathrm{r}}$. {Integrating the importance sampling scheme into the {MLMC} estimator ensured that the constant associated with its complexity also reduced significantly compared with the {MLMC} estimator for smooth, nonrare observables introduced in \citep{mlmc_mvsde}.}

Future studies could include extending the importance sampling scheme to higher-dimensional problems, using model reduction techniques or stochastic gradient-based learning methods to solve the associated higher-dimensional stochastic optimal control problem (see Remark~\ref{rem:hjb_dimension}). The importance sampling scheme could be further improved by solving an optimal control problem minimizing the level-difference estimator variance rather than the single-level DLMC estimator (see Remark~\ref{rem:mlmc_variance}). The {multilevel DLMC} algorithm could be optimized for determining the optimal parameters $\tau$ and $\theta$ \citep{optimal_mlmc} or integrating a continuation {MLMC} algorithm~\citep{cont_mlmc} (see Remark~\ref{rem:optimal_parameters}).
The present analysis could be extended to numerically address non-Lipschitz rare-event observables, such as the indicator function, to compute rare-event probabilities. Multiple discretization parameters for the decoupled MV-SDE ($P, N$) suggest extending the current work to a multi-index Monte Carlo~\citep{tempone_mimc,mlmc_mvsde} setting to further reduce computational complexity. 

\appendix
\label{sec:appendix}

\section{{Computational cost} of the Multilevel Double Loop Monte Carlo Estimator \eqref{eqn:mldlmc_estimator}}
\label{app:work}
{
For a given level $\ell$, we estimate $\E{\Delta G_\ell} = \E{G_\ell - G_{\ell-1}}$. First, we derive the {computational cost} of estimating $\E{G_\ell}$ using DLMC Algorithm~\ref{alg:dlmc_revised}. We consider the discretization $0=t_0<t_1<t_2<\ldots<t_{N_\ell} = T$ of the time domain $[0, T]$ with $N_\ell$ equal time steps (i.e., $t_n = n \Delta t, \quad n=0,1,\ldots,N_\ell$ and $\Delta t = T/{N_\ell}$). First, we simulate the particle system at level $\ell$. For each particle $p \in 1,\ldots,P_\ell$ and each time step $n \in 1,\ldots,N_\ell$, the Euler--Maruyama time-stepping scheme is written as

\begin{align}
	\label{eqn:particle_system_euler}
	X_p^{P_\ell \mid N_\ell}(t_{n+1}) &= X_p^{P_\ell \mid N_\ell}(t_{n}) + b \left( X_p^{P_\ell \mid N_\ell}(t_n), \frac{1}{P_\ell} \sum_{j=1}^{P_\ell} \kappa_1 \left(X_p^{P_\ell \mid N_\ell}(t_n), X_j^{P_\ell \mid N_\ell}(t_n) \right) \right) \Delta t \nonumber \\
	&\quad + \sigma \left( X_p^{P_\ell \mid N_\ell}(t_n), \frac{1}{P_\ell} \sum_{j=1}^{P_\ell} \kappa_2 \left(X_p^{P_\ell \mid N_\ell}(t_n), X_j^{P_\ell \mid N_\ell}(t_n) \right) \right) \Delta W(t_n), \quad n = 1,\ldots,N_\ell ,\\
	 X_p^{P_\ell \mid N_\ell}(t_0) &\sim \mu_0 \cdot \nonumber
\end{align}

\begin{algorithm} 
    \caption{General double loop Monte Carlo algorithm for decoupled McKean--Vlasov Stochastic Differential Equation}
\label{alg:dlmc_revised}
    \SetAlgoLined
    \textbf{Inputs: } $P,N,M_{1},M_2$; \\
    \For{$m_1=1,\ldots,M_1$}{
    Generate $\mu^{P|N} \left(\omega_{1:P}^{(m_1)} \right)$ with $P$-particle system and $N$ time steps using \eqref{eqn:particle_system_euler}; \\
    \For{$m_2=1,\ldots,M_2$}{
    Given $\mu^{P|N} \left(\omega_{1:P}^{(m_1)} \right)$, generate sample path of decoupled MV-SDE with $N$ time steps with $\bar{\omega}^{(m_2)}$ using \eqref{eqn:dmvsde_euler};\\
    Compute $G\left(\Bar{X}^{P|N}(T)\right) \left(\omega_{1:P}^{(m_1)}, \bar{\omega}^{(m_2)}\right)$; \\
    }
    }
    Approximate $\E{G\left(\Bar{X}^{P|N}(T)\right)}$ by $\frac{1}{M_1} \sum_{m_1=1}^{M_1} \frac{1}{M_2} \sum_{m_2=1}^{M_2} G\left(\Bar{X}^{P|N}(T)\right) \left(\omega_{1:P}^{(m_1)}, \bar{\omega}^{(m_2)}\right)$ ;
\end{algorithm}

The {computational cost} per time step per particle is $\order{P_\ell^{\gamma_p}}$ because the cost of computing the empirical average $\frac{1}{P_\ell} \sum_{j=1}^{P_\ell} \kappa \left(X_p^{P_\ell \mid N_\ell}(t_n), X_j^{P_\ell \mid N_\ell}(t_n) \right)$ in the drift and diffusion coefficients is assumed to be $\order{P_\ell^{\gamma_p}}$ for $\gamma_p > 0$. This cost is $\order{P_\ell}$ for a naïve method. Hence, the {computational cost} of simulating a $P_\ell$ particle system once using $N_\ell$ time steps with scheme~\eqref{eqn:particle_system_euler} is $\order{N_\ell P_\ell^{1+\gamma_p}}$. For a given realization of the particle system, we simulate the decoupled MV-SDE~\eqref{eqn:decoupled_mvsde} using the Euler--Maruyama scheme with the same time discretization as above. For each time step $n \in 1,\ldots,N_\ell$, the time-stepping scheme is written as

\begin{align}
	\label{eqn:dmvsde_euler}
	\bar{X}_\zeta^{P_\ell \mid N_\ell}(t_{n+1}) &= \bar{X}_\zeta^{P_\ell \mid N_\ell}(t_{n}) + b \left( \bar{X}_\zeta^{P_\ell \mid N_\ell}(t_n), \frac{1}{P_\ell} \sum_{j=1}^{P_\ell} \kappa_1 \left(\bar{X}_\zeta^{P_\ell \mid N_\ell}(t_n), X_j^{P_\ell \mid N_\ell}(t_n) \right) \right) \Delta t \nonumber \\
	&\quad + \sigma \left( \bar{X}_\zeta^{P_\ell \mid N_\ell}(t_n), \frac{1}{P_\ell} \sum_{j=1}^{P_\ell} \kappa_2 \left(\bar{X}_\zeta^{P_\ell \mid N_\ell}(t_n), X_j^{P_\ell \mid N_\ell}(t_n) \right) \right) \Delta \bar{W}(t_n), \quad n = 1,\ldots,N_\ell, \\
	 \bar{X}_\zeta^{P_\ell \mid N_\ell}(t_0) &\sim \mu_0 \cdot \nonumber
\end{align}

The {computational cost} per time step is again $\order{P_\ell^{\gamma_p}}$ due to the cost of computing the empirical average in the drift and diffusion coefficients. Hence, the {computational cost} of simulating the decoupled MV-SDE~\eqref{eqn:decoupled_mvsde} using the above scheme~\eqref{eqn:dmvsde_euler} is $\order{N_\ell P_\ell^{\gamma_p}}$. Thus, the {computational complexity} of DLMC in Algorithm~\ref{alg:dlmc_revised} per level can be written as follows:

\begin{align}
\label{eqn:work_per_level}
	\mathcal{W}_\ell &= \order{M_{1,\ell} \left( N_\ell P_\ell^{1+\gamma_p} + M_{2,\ell} \left( N_\ell P_\ell^{\gamma_p} \right) \right)} \nonumber \\
	&= \order{M_{1,\ell} N_\ell P_\ell^{1+\gamma_p} + M_{1,\ell} M_{2,\ell} N_\ell P_\ell^{\gamma_p}} \cdot
\end{align}

We obtain the expression in~\eqref{eqn:mldlmc_work} by generalizing this cost estimate for any time-stepping scheme (with a {computational complexity} of $\order{N_\ell^{\gamma_n}}$ instead of $\order{N_\ell}$) and then summing over all levels. 
}
\section{Proof of Theorem \ref{th:mldlmc_complexity}}
\label{app:1}

For given level $L \in \mathbb{N}$, \eqref{eqn:mldlmc_actual_samples} provides the optimal number of samples to satisfy the variance constraint in \eqref{eqn:mldlmc_var_constraint} for the proposed {multilevel DLMC} estimator. We bound the {multilevel DLMC} estimator \eqref{eqn:mldlmc_estimator} cost as follows: 

\begin{align}
    \mathcal{W}[\mathcal{A}_{\textrm{MLMC}}] &\lesssim \sum_{\ell=0}^L \left({M_{1,\ell}} P_\ell^{1+\gamma_p} N_\ell^{\gamma_n} + {M_{1,\ell} M_{2,\ell} } P_\ell^{\gamma_p} N_\ell^{\gamma_n}\right) \nonumber \\
    &\lesssim \sum_{\ell=0}^L \left({(\mathcal{M}_{1,\ell}+1)} P_\ell^{1+\gamma_p} N_\ell^{\gamma_n} + {(\mathcal{M}_{1,\ell}+1) \left(\frac{\tilde{\mathcal{M}}_\ell}{\lceil \mathcal{M}_{1,\ell} \rceil}+1 \right)} P_\ell^{\gamma_p} N_\ell^{\gamma_n}\right) \nonumber \\
    &\leq \underbrace{\sum_{\ell=0}^L \left(\mathcal{M}_{1,\ell} P_\ell^{1+\gamma_p} N_\ell^{\gamma_n} + \tilde{\mathcal{M}}_{\ell} P_\ell^{\gamma_p} N_\ell^{\gamma_n}\right)}_{=W_1,\text{objective function of \eqref{eqn:mldlmc_optim}}} \nonumber \\
    \label{eqn:mldlmc_work_step1}
    &+ \underbrace{\sum_{\ell=0}^L \left(P_\ell^{1+\gamma_p} N_\ell^{\gamma_n} + P_\ell^{\gamma_p} N_\ell^{\gamma_n}\right)}_{W_2,\text{cost of one realization per level}} + \underbrace{{\sum_{\ell=0}^L \mathcal{M}_{1,\ell} P_\ell^{\gamma_p} N_\ell^{\gamma_n}}}_{=W_3} + \underbrace{{\sum_{\ell=0}^L \frac{\tilde{\mathcal{M}}_\ell}{\lceil \mathcal{M}_{1,\ell} \rceil} P_\ell^{\gamma_p} N_\ell^{\gamma_n}}}_{=W_4} \cdot
\end{align}

Because $P_\ell > 1$ and $\gamma_p > 0$, $W_3$ is always dominated by the first term in $W_1$ (i.e., $\sum_{\ell=0}^L \mathcal{M}_{1,\ell} P_\ell^{1+\gamma_p} N_\ell^{\gamma_n}$). We analyze each term individually. By substituting \eqref{eqn:mldlmc_optimal_samples} in $W_1$,

\begin{equation}
    \label{eqn:mldlmc_work_step2}
    W_1 = \frac{C_\nu^2}{(1-\theta)^2 \tol_\mathrm{r}^2 \abs{\E{G}}^2} \left(\sum_{\ell=0}^L \sqrt{P_\ell^{\gamma_p} N_\ell^{\gamma_n}}(\sqrt{V_{1,\ell}P_\ell} + \sqrt{V_{2,\ell}})\right)^2 \cdot
\end{equation}

By selecting level $L$ that satisfies \eqref{eqn:mldlmc_bias_constraint} and using Assumption~\ref{ass:mldlmc_bias},

\begin{equation}
    \label{eqn:mldlmc_work_step3}
    \abs{\mathbb{E}[G-G_L]} \leq C_b \tau^{-\Tilde{\alpha} L} \approx \theta\tol_\mathrm{r} \abs{\E{G}},
\end{equation}
and hence,
\begin{equation}
    \label{eqn:mldlmc_optim_level}
    L = \left\lceil \frac{1}{\Tilde{\alpha}}\log(\frac{C_b}{\theta \abs{\E{G}}}\tol_\mathrm{r}^{-1}) \right\rceil \cdot
\end{equation}

Using the hierarchies \eqref{eqn:mldlmc_hierarchy} and Assumption~\ref{ass:mldlmc_var} in \eqref{eqn:mldlmc_work_step2}, 

\begin{equation}
    \label{eqn:mldlmc_work_step4}
    W_1 \lesssim {\tol_\mathrm{r}^{-2}\left(\sum_{\ell=0}^L \tau^{\frac{1+\gamma_p+\gamma_n-\Tilde{w}}{2}\ell} + \tau^{\frac{\gamma_p+\gamma_n-\Tilde{s}}{2}\ell}\right)^2} \cdot
\end{equation}

The summation in \eqref{eqn:mldlmc_work_step4} has two terms; thus, we have the following two cases:

\begin{itemize}
    \item \textbf{Case 1: } $\Tilde{s}+1 \geq \Tilde{w}$, i.e., the first term dominates, and $W_1$ can be expressed as 
    \begin{equation}
        \label{eqn:mldlmc_work_case1_v1}
        W_1 \lesssim {\tol_\mathrm{r}^{-2}\left(\sum_{\ell=0}^L \tau^{\frac{1+\gamma_p+\gamma_n-\Tilde{w}}{2}\ell}\right)^2} , 
    \end{equation}
and simplified using \eqref{eqn:mldlmc_optim_level} and the sum of a geometric series,
    \begin{equation}
        \label{eqn:mldlmc_work_case1_v2}
        W_1 \lesssim \begin{cases}
            {\tol_\mathrm{r}^{-2}} &\quad,{\Tilde{w}>1+\gamma_p+\gamma_n} \\
            {\tol_\mathrm{r}^{-2}(\log\tol_\mathrm{r}^{-1})^2} &\quad, {\Tilde{w}=1+\gamma_p+\gamma_n} \\
            {\tol_\mathrm{r}^{-2-\left(\frac{1+\gamma_p+\gamma_n-\Tilde{w}}{\Tilde{\alpha}}\right)}} &\quad,{\Tilde{w}<1+\gamma_p+\gamma_n}
        \end{cases} , 
    \end{equation}
which  can be expressed more compactly:
    \begin{equation}
        \label{eqn:mldlmc_work_case1_v3}
        W_1 \lesssim {\tol_\mathrm{r}^{-2-\max\left(0,\frac{1+\gamma_p+\gamma_n-\Tilde{w}}{\Tilde{\alpha}}\right)}(\log\tol_\mathrm{r}^{-1})^{2 \mathcal{J}_1 }} ,
    \end{equation}
 {
    where 
    \begin{equation*}
    	\mathcal{J}_1 = \begin{cases}
    	1, &\quad \text{if } \tilde{w}=1+\gamma_p+\gamma_n, \\
    	0, &\quad \text{else}.
    	\end{cases}
    \end{equation*}
    }

    \item \textbf{Case 2: } $\Tilde{s}+1 < \Tilde{w}$, i.e., the second term dominates, and express $W_1$ as 
    \begin{equation}
        \label{eqn:mldlmc_work_case2_v1}
        W_1 \lesssim {\tol_\mathrm{r}^{-2}\left(\sum_{\ell=0}^L \beta^{\frac{\gamma_p+\gamma_n-\Tilde{s}}{2}\ell}\right)^2} , 
    \end{equation}
and simplify using the sum of a geometric series,
    \begin{equation}
        \label{eqn:mldlmc_work_case2_v2}
        W_1 \lesssim \begin{cases}
            {\tol_\mathrm{r}^{-2}} &\quad,{\Tilde{s}>\gamma_p+\gamma_n} \\
            {\tol_\mathrm{r}^{-2}(\log\tol_\mathrm{r}^{-1})^2} &\quad, {\Tilde{s}=\gamma_p+\gamma_n} \\
            {\tol_\mathrm{r}^{-2-\left(\frac{\gamma_p+\gamma_n-\Tilde{s}}{\Tilde{\alpha}}\right)}} &\quad,{\Tilde{s}<\gamma_p+\gamma_n}
        \end{cases} , 
    \end{equation}
which can be expressed more compactly:
    \begin{equation}
        \label{eqn:mldlmc_work_case2_v3}
        W_1 \lesssim {\tol_\mathrm{r}^{-2-\max\left(0,\frac{\gamma_p+\gamma_n-\Tilde{s}}{\Tilde{\alpha}}\right)}(\log\tol_\mathrm{r}^{-1})^{2 \mathcal{J}_2 }} ,
    \end{equation}
    {
    where 
    \begin{equation*}
    	\mathcal{J}_2 = \begin{cases}
    	1, &\quad \text{if } \tilde{s}=\gamma_p+\gamma_n, \\
    	0, &\quad \text{else}.
    	\end{cases}
    \end{equation*}}
    In this case, $W_3$ is of a lower order than the first term in $W_1$; therefore, $W_1$ dominates as $\tol_\mathrm{r} \rightarrow 0$.
  
\end{itemize}

Using \eqref{eqn:mldlmc_hierarchy}, for $W_2$, 
\begin{align}
    \label{eqn:mldlmc_work_step5}
    W_2 &\lesssim {\tol_\mathrm{r}^{-\left(\frac{1+\gamma_p+\gamma_n}{\Tilde{\alpha}}\right)}} \cdot
\end{align}

Next, we examine $W_4$:

\begin{align}
		W_4 &= \sum_{\ell=0}^L \frac{\tilde{\mathcal{M}}_\ell}{\lceil \mathcal{M}_{1,\ell} \rceil} P_\ell^{\gamma_p} N_\ell^{\gamma_n} \nonumber \\
		&\leq \sum_{\ell=0}^L \frac{\tilde{\mathcal{M}}_\ell}{\max\{1,\mathcal{M}_{1,\ell}\}} P_\ell^{\gamma_p} N_\ell^{\gamma_n} \nonumber\\
		&= \sum_{\ell=0}^L \min\{\tilde{\mathcal{M}}_\ell P_\ell^{\gamma_p} N_\ell^{\gamma_n}, \frac{\tilde{\mathcal{M}}_\ell}{\mathcal{M}_{1,\ell}}P_\ell^{\gamma_p} N_\ell^{\gamma_n}\} \nonumber\\
		\label{mldlmc_work_step6}
		&\leq \sum_{\ell=0}^L \tilde{\mathcal{M}}_\ell P_\ell^{\gamma_p} N_\ell^{\gamma_n} \cdot
\end{align}

$W_4$ has the same or lower order than that of $W_1$. Next, we must ensure that $W_1$ dominates $W_2$ as $\tol_\mathrm{r} \rightarrow 0$ for the proposed {multilevel DLMC} method to be feasible. Comparing \eqref{eqn:mldlmc_work_step5} to $W_1$ for the cases, the following condition ensures $W_1$ is the dominant term:

\begin{equation}
    \label{eqn:mldlmc_work_step7}
    \tilde{\alpha} \geq \frac{1}{2} \min (\tilde{w}, 1+\tilde{s}, 1+\gamma_p+\gamma_n)  \cdot
\end{equation}

Thus, \eqref{eqn:mldlmc_work_step7}, \eqref{eqn:mldlmc_work_case1_v3}, and \eqref{eqn:mldlmc_work_case2_v3} complete the proof.

\section{Estimating level differences for the multilevel double loop Monte Carlo estimator}
\label{app:2}

\begin{algorithm}[H] 
    \caption{Importance sampling scheme to estimate $\E{\Delta G_\ell}$ using antithetic sampler}
    \label{alg:mldlmc_ld}
    \SetAlgoLined
    \textbf{Inputs: } $\ell,M_1,M_2,\zeta(\cdot,\cdot)$; \\
    \For{$m_1=1,\ldots,M_1$}{
        Generate $\mu_\ell(\omega_{1:P_\ell}^{(\ell,m_1)})$ using \eqref{eqn:dmvsde_discrete_law}; \\
        \For{$a=1,\ldots,\tau$}{
            Generate $\mu_{\ell-1}^{(a)}(\omega_{(a-1)P_{\ell-1}+1:aP_{\ell-1}}^{(\ell,m_1)})$ using \eqref{eqn:dmvsde_discrete_law}; 
        } 
        \For{$m_2=1,\ldots,M_2$}{
            Given $\mu_\ell(\omega_{1:P_\ell}^{(\ell,m_1)})$ and $\zeta(\cdot,\cdot)$, generate sample path of \eqref{eqn:dmvsde_sde_is} at level $\ell$ with $\bar{\omega}^{(\ell,m_2)}$;\\
            Compute $G_\ell^\mathrm{IS}(\omega_{1:P_\ell}^{(\ell,m_1)},\bar{\omega}^{(\ell,m_2)})$; \\
            \For{$a=1,\ldots,\tau$}{
                Given $\mu_{\ell-1}^{(a)}(\omega_{(a-1)P_{\ell-1}+1:aP_{\ell-1}}^{(\ell,m_1)})$ and $\zeta(\cdot,\cdot)$, generate sample path of $\eqref{eqn:dmvsde_sde_is}$ at level $\ell-1$ with $\bar{\omega}^{(\ell,m_2)}$; \\
                Compute $G_{\ell-1}^\mathrm{IS}(\omega_{(a-1)P_{\ell-1}+1:aP_{\ell-1}}^{(\ell,m_1)},\bar{\omega}^{(\ell,m_2)})$; \\  
            }
            $\hat{\mathcal{G}}_{\ell-1}^\mathrm{IS}(\omega_{1:P_\ell}^{(\ell,m_1)},\bar{\omega}^{(\ell,m_2)}) = \frac{1}{\tau} \sum_{a=1}^{\tau} G_{\ell-1}^\mathrm{IS}(\omega_{(a-1)P_{\ell-1}+1:aP_{\ell-1}}^{(\ell,m_1)},\bar{\omega}^{(\ell,m_2)})$;\\
        }
        $\Delta G_\ell^{(m_1,m_2)} = (G_\ell^\mathrm{IS} - \hat{\mathcal{G}}_{\ell-1}^\mathrm{IS})(\omega_{1:P_\ell}^{(\ell,m_1)},\bar{\omega}^{(\ell,m_2)})$;\\
        }
    Approximate $\E{G_{\ell} -G_{\ell-1}}$ by $\frac{1}{M_1} \sum_{m_1=1}^{M_1} \frac{1}{M_2} \sum_{m_2=1}^{M_2} \Delta G_\ell^{(m_1,m_2)}$; \\
\end{algorithm}

\section{Estimating variances for the adaptive multilevel double loop Monte Carlo algorithm}
\label{app:3}

\begin{algorithm}[H] 
    \label{alg:mldlmc_est_variance}
    \caption{Estimating $V_{1,\ell}$ and $V_{2,\ell}$ for adaptive {multilevel DLMC}}
    \SetAlgoLined
    \textbf{Inputs: } $\ell,M_1,M_2,\zeta(\cdot,\cdot)$; \\
    \For{$m_1=1,\ldots,M_1$}{
        Generate $\mu_\ell(\omega_{1:P_\ell}^{(\ell,m_1)})$ using \eqref{eqn:dmvsde_discrete_law}; \\
        \For{$a=1,\ldots,\tau$}{
            Generate $\mu_{\ell-1}^{(a)}(\omega_{(a-1)P_{\ell-1}+1:aP_{\ell-1}}^{(\ell,m_1)})$ using \eqref{eqn:dmvsde_discrete_law}; 
        } 
        \For{$m_2=1,\ldots,M_2$}{
            Given $\mu_\ell(\omega_{1:P_\ell}^{(\ell,m_1)})$ and $\zeta(\cdot,\cdot)$, generate sample path of \eqref{eqn:decoupled_mvsde} at level $\ell$ with $\bar{\omega}^{(\ell,m_2)}$;\\
            Compute $G_\ell^\mathrm{IS}(\omega_{1:P_\ell}^{(\ell,m_1)},\bar{\omega}^{(\ell,m_2)})$; \\
            \For{$a=1,\ldots,\tau$}{
                Given $\mu_{\ell-1}^{(a)}(\omega_{(a-1)P_{\ell-1}+1:aP_{\ell-1}}^{(\ell,m_1)})$ and $\zeta(\cdot,\cdot)$, generate sample path of $\eqref{eqn:decoupled_mvsde}$ at level $\ell-1$ with $\bar{\omega}^{(\ell,m_2)}$; \\
                Compute $G_{\ell-1}^\mathrm{IS}(\omega_{(a-1)P_{\ell-1}+1:aP_{\ell-1}}^{(\ell,m_1)},\bar{\omega}^{(\ell,m_2)})$; \\  
            }
            $\hat{\mathcal{G}}_{\ell-1}^\mathrm{IS}(\omega_{1:P_\ell}^{(\ell,m_1)},\bar{\omega}^{(\ell,m_2)}) = \frac{1}{\tau} \sum_{a=1}^{\tau} G_{\ell-1}^\mathrm{IS}(\omega_{(a-1)P_{\ell-1}+1:aP_{\ell-1}}^{(\ell,m_1)},\bar{\omega}^{(\ell,m_2)})$;\\
        }
        $\Delta G_\ell^{(m_1,m_2)} = (G_\ell^\mathrm{IS} - \hat{\mathcal{G}}_{\ell-1}^\mathrm{IS})(\omega_{1:P_\ell}^{(\ell,m_1)},\bar{\omega}^{(\ell,m_2)})$;\\
        Approximate $\E{\Delta G_\ell \mid \{\mu_\ell,\mu_{\ell-1}\}(\omega_{1:P_\ell}^{(\ell,m_1)})}$ by $\frac{1}{M_2} \sum_{m_2=1}^{M_2} \Delta G_\ell^{(m_1,m_2)}$; \\
        Approximate $\Var{\Delta G_\ell \mid \{\mu_\ell,\mu_{\ell-1}\}(\omega_{1:P_\ell}^{(\ell,m_1)})}$ by sample variance of $\left\lbrace \Delta G_\ell^{(m_1,m_2)} \right\rbrace_{m_2=1}^{M_2}$ ; \\
    }
    Approximate $V_{1,\ell}$ by sample variance of  $\left\lbrace \E{\Delta G_\ell \mid \{\mu_\ell,\mu_{\ell-1}\}(\omega_{1:P_\ell}^{(\ell,m_1)})} \right\rbrace_{m_1=1}^{M_1} $; \\
    Approximate $V_{2,\ell}$ by $\frac{1}{M_1} \sum_{m_1=1}^{M_1} \Var{\Delta G_\ell \mid \{\mu_\ell,\mu_{\ell-1}\}(\omega_{1:P_\ell}^{(\ell,m_1)})}$.
\end{algorithm}

\bibliography{References}
\bibliographystyle{plainnat}
\setcitestyle{authoryear,open={[},close={]}}
\end{document}